\documentclass[12pt]{amsart}
\usepackage{enumitem,graphicx,color,bbm}
\usepackage[pdftex, colorlinks]{hyperref}
\hypersetup{colorlinks=true,  
    linkcolor=black,          
    citecolor=black,        
    filecolor=black,      
    urlcolor=black           
    }

\newtheorem{thm}{Theorem}
\newtheorem{lem}[thm]{Lemma}

\newtheorem{cor}[thm]{Corollary}

\newtheorem*{thm*}{Theorem}

\theoremstyle{remark}
\newtheorem{remark}{Remark}

\DeclareMathOperator{\lin}{lin}
\DeclareMathOperator{\ord}{ord}

\DeclareMathOperator{\dimn}{dim}
\DeclareMathOperator{\essinf}{ess\, inf}
\DeclareMathOperator{\esssup}{ess\, sup}

\theoremstyle{remark}
\newtheorem*{rem}{Remark}

\newcommand{\R}{\mathbb R}
\newcommand{\C}{\mathbb C}
\newcommand{\N}{\mathbb N}
\newcommand{\Z}{\mathbb Z}
\newcommand{\LL}{\mathcal L}
\newcommand{\PP}{\mathbb P}
\newcommand{\NN}{\mathcal N}

\newcommand{\om}{\omega}
\newcommand{\mc}[1]{\mathcal{#1}}
\newcommand{\Om}{\Omega}
\newcommand{\bbp}{\mathbb{P}}
\newcommand{\sig}{\sigma}
\newcommand{\mcl}{\mathcal{L}}
\newcommand{\lam}{\lambda}
\newcommand{\ka}{\kappa}
\newcommand{\lex}[2]{\lam_{#2}({#1})}
\newcommand{\al}{\alpha}
\newcommand{\Blaschke}{\mathsf{Blaschke_R(\Omega)}}

\begin{document}
\title[Lyapunov spectrum stability and collapse]{Stability and Collapse of the Lyapunov 
spectrum for Perron-Frobenius Operator cocycles}
\author{Cecilia Gonz\'alez-Tokman and Anthony Quas}
\begin{abstract}
In this paper, we study random Blaschke products, acting on the
unit circle, and consider the cocycle of Perron-Frobenius operators acting on 
Banach spaces of
analytic functions on an annulus. We completely describe the Lyapunov spectrum
of these cocycles.
As a corollary, we obtain a simple random Blaschke product system
where the Perron-Frobenius cocycle has infinitely many distinct Lyapunov exponents, but where
arbitrarily small natural perturbations cause a complete collapse of the Lyapunov spectrum,
except for the exponent 0 associated with the absolutely continuous invariant 
measure. That is, under perturbations, the Lyapunov exponents become 0 with
multiplicity 1, and $-\infty$ with infinite multiplicity.
This is superficially similar to the finite-dimensional phenomenon,
discovered by Bochi \cite{Bochi-thesis}, that away from the uniformly hyperbolic setting,
small perturbations can lead to a collapse of the Lyapunov spectrum to zero.
In this paper, however, the cocycle and its perturbation are explicitly described; and further, 
the mechanism for collapse is quite different.

We study stability of the Perron-Frobenius cocycles
arising from general random Blaschke products. We give a necessary and sufficient
criterion for stability of the Lyapunov spectrum in terms of the derivative
of the random Blaschke product at its random fixed point, and use this to show
that an open dense set of Blaschke product cocycles have hyperbolic 
Perron-Frobenius cocycles.

In the final part, we prove a relationship between the Lyapunov spectrum of a single
cocycle acting on two different Banach spaces, allowing us to draw conclusions
for the same cocycles acting on $C^r$ functions spaces.
\end{abstract}
\maketitle

\section{Introduction}
A well known technique for computing rates of decay of correlation for hyperbolic
dynamical systems is based on identifying the spectral gap for its Perron-Frobenius
operator, acting on a suitable Banach space of functions (different function spaces 
give rise to different spectral gaps, and so to different rates of decay). In particular, 
exceptional eigenvalues (those outside the essential spectral radius) play a key role 
in determining rates of decay of correlation. 

In the autonomous case, Keller and Liverani \cite{KellerLiverani} established, in a landmark
paper based on spectral techniques, robust checkable conditions under which the 
exceptional eigenvalues, and their corresponding eigenspaces, vary continuously in
response to small perturbations of the dynamics, or more precisely, of its Perron-Frobenius operator. 

With Froyland,
we have been working towards random analogues of the result of Keller and Liverani, where instead of 
the Perron-Frobenius operator of a single map, one has a random dynamical system
(that is, a skew product of the \emph{base dynamics} $\sigma:\Omega\to\Omega$
with $\omega$-dependent maps, $T_\omega$, in the fibres). One then forms
the random linear dynamical system, or cocycle, where the fibre maps are now the Perron-Frobenius
operators, $\LL_{T_\omega}$, acting on a Banach $X$. In the case of cocycles, 
spectral techniques no longer apply. 
A significant motivation for this line of research comes from the fact that Froyland 
and collaborators have developed 
effective tools making use of finite-dimensional approximations of Perron-Frobenius cocycles
to identify regions of interest in environmental dynamical systems (see e.g. \cite{SFM}).
Our program is aimed towards assessing the robustness of these methods.

The non-autonomous analogue of eigenvalues, Lyapunov exponents, are known
to be far more sensitive to perturbations than eigenvalues. Based on an outline proposed by Ma\~n\'e, 
Bochi showed in his thesis \cite{Bochi-thesis}
 that on any compact surface there is a residual set of $C^1$ area-preserving 
 diffeomorphisms that are either Anosov or have all Lyapunov exponents equal to zero.
Similar results are established for two-dimensional matrix cocycles.
Bochi and Viana \cite{BochiViana} extended this to higher-dimensional systems and
cocycles, so that these results show that Lyapunov exponents are highly unstable. 

On the other hand, Ledrappier and Young \cite{LedrappierYoung} showed that if one makes
\emph{absolutely continuous} perturbations to invertible matrix cocycles, 
small perturbations lead to small changes in the Lyapunov exponents. Ochs \cite{Ochs} 
then showed in the finite-dimensional invertible case, small changes in Lyapunov 
exponents lead
to small changes (in probability) in the Oseledets spaces, the non-autonomous 
analogue of (generalized) eigenspaces.

In \cite{FGQ-Nonlin}, with Froyland, we gave fairly general conditions under which
the top Oseledets space responds continuously to perturbations of the Perron-Frobenius operator
cocycle. In \cite{FGQ-CPAM}, we extended the results of Ledrappier and Young, and of Ochs
to the semi-invertible setting (where the base dynamics are assumed to be invertible,
but no injectivity or invertibility assumptions are made on the operators). 
In \cite{FGQ-TAMS}, we (again with Froyland) also gave the first result
on stability of Lyapunov exponents in an infinite-dimensional setting:
we consider
Hilbert-Schmidt cocycles on a separable Hilbert space with exponential decay
of the entries. In this case, there is no single natural notion of noise, but we consider
perturbations by additive noise with faster exponential decay. Again, we recover
in \cite{FGQ-TAMS} the stability of the Lyapunov exponents and Oseledets spaces.
Another related result is due to Nakano and Wittsten \cite{NakanoWittsten}, who study 
random perturbations of 
partially expanding maps of the torus (skew products of circle rotations over uniformly 
expanding maps of the circle) whose transfer operators have a spectral gap. Using 
semi-classical analysis, they show that the spectral 
gap is preserved under small random perturbations.

One could interpret the available results collectively as saying that carefully chosen 
perturbations may lead to radical change to
the Lyapunov spectrum, while noise-like perturbations of the cocycle tend to lead to
small changes to the Lyapunov spectrum.

The examples that we focus on in this work are expanding finite Blaschke products,  a class of analytic maps
from the unit circle to itself. The Perron-Frobenius operators for single maps of this
type, acting on the Hardy space of a suitable annulus, were studied by Bandtlow,
Just and Slipantschuk \cite{BJS}, where they used results on composition operators
to obtain a precise description 
of the set of eigenvalues. Indeed, the eigenvalues they obtain are precisely the non-negative
powers of the derivative of the underlying Blaschke product at its unique (attracting) fixed
point in the unit disc, and their complex conjugates.
We study a random version, where instead of a single 
Blaschke product, $T$, one applies a Blaschke product $T_\omega$ that is 
selected by the base dynamics.

In Section~\ref{sec:LyapSpectrum}
we generalize the results of \cite{BJS}
to the non-autonomous setting, and show that the Lyapunov spectrum 
of the Perron-Frobenius cocycle is given by the non-negative multiples of the
Lyapunov exponent of the underlying Blaschke
product cocycle at the random attracting fixed point in the unit disc (with multiplicity
two for all positive multiples of this exponent). To our knowledge, this is the first time
that a complete description of an infinite Lyapunov spectrum of a Perron-Frobenius cocycle
has been given. We find it quite remarkable that the Perron-Frobenius spectrum 
(describing what happens to densities on the unit circle) is governed in this way 
by the derivative at the random fixed point in the interior of the unit disc.

In Section \ref{sec:phasetrans}, 
we show, to our surprise, that there are natural examples of random dynamical systems,
where natural perturbations of the Perron-Frobenius cocycle
lead to a collapse of the Lyapunov spectrum.
We focus on a particular Blaschke
product cocycle (with maps $T_0$ and $T_1$ applied in an i.i.d.\ way,
where $T_0(z)=z^2$ and $T_1(z)=[(z+\frac 14)/(1+\frac z4)]^2$). 
The Perron-Frobenius operator of $T_0$ is known to be
highly degenerate \cite[Exercise 2.14]{Baladi}. If the frequency of applying $T_0$
is $p$, then we find a phase transition: for $p\ge \frac 12$, the Lyapunov spectrum
collapses (so that there is an exponent 0 with multiplicity 1, and all other Lyapunov exponents
are $-\infty$), while for $p<\frac 12$, there is a complete (infinite) Lyapunov spectrum. We then
consider normal perturbations (corresponding to adding random normal noise
to the dynamical system), $\LL^\epsilon_0$ and $\LL^\epsilon_1$ and show that
for $p\ge \frac 14$, there is collapse of the Lyapunov spectrum for all $\epsilon>0$.

In particular, for $\frac 14\le p<\frac 12$, the unperturbed system has a complete
Lyapunov spectrum, while arbitrarily small normal perturbations have a collapsed
Lyapunov spectrum. We also show that collapse can occur for every $p>0$ 
in the setting of uniform perturbations. Unlike in Bochi's setting, our perturbations
are explicitly described and arise naturally in the area. 

In Section \ref{sec:Lyapstab}, we show that, under natural conditions on the underlying Blaschke
product cocycle, the corresponding Perron-Frobenius cocycle has some hyperbolicity properties,
guaranteeing a uniform separation between fast and slow Oseledets subspaces. This allows us
to give a simple necessary and sufficient condition
for stability of the Lyapunov spectrum.
In particular, we show that the set of stable cocycles is open and dense. 

In the Appendix we compare Lyapunov exponents and Oseledets splittings of random linear dynamical
systems that arise from  restricting the cocycle to a finer subspace. 
This allows us to study the action of the Perron-Frobenius cocycle on coarser Banach spaces
such as $C^r$.
(In the case of autonomous systems, this question has been addressed 
by Baladi and Tsujii in \cite{BaladiTsujii08}.)
We  construct an explicit example of a cocycle of analytic maps of the circle 
(indeed, of Blaschke products) whose Perron-Frobenius operator, when acting on a 
$C^r$ function space, has negative exceptional Lyapunov exponents. 
The corresponding situation in the autonomous case was first established by Keller and 
Rugh \cite{KellerRugh}, and it was suggested in \cite{BJS} that Blaschke products 
could also exhibit this phenomenon. 

\section{Statement of Theorems}
In this section, we state our main theorems. We will assume standard definitions, 
but for completeness, some terms used here will be defined in the next section.

For a finite Blaschke product,  set $r_T(R):=\max_{|z|=R}|T(z)|$.
By \cite{Tischler} if the restriction of a Blaschke product to the unit circle is expanding,
then $r_T(R)<R$ for some $R<1$. 
By the Hadamard three circle theorem, $R\mapsto r_T(R)$ satisfies the log-convexity property:
$r_T(R_1^{1-\lambda}R_2^\lambda)\le r_T(R_1)^{1-\lambda}r_T(R_2)^\lambda$
for $0<\lambda< 1$.
In particular for Blaschke products $r_T(1)=1$
so if $r_T(R)<R$ then $r_{T}(R')<R'$ for all $R'\in(R,1)$.
We let $\beta(T)=\inf\{R>0\colon r_T(R)<R\}$. Then 
$r_T(R)<R$ if and only if $R\in (\beta(T),1)$.  

For a Blaschke product cocycle $\mathcal T=(T_\omega)_{\omega\in\Omega}$, 
we define $r_{\mathcal T}(R)=\esssup_{\omega\in\Omega} r_{T_\omega}(R)$.
If $\mathcal T$ is finite, for $R\in (\esssup_\omega \beta(\omega),1)$ 
we have $r_{\mathcal T}(R)<R$. 
The condition $r_{\mathcal T}(R)<R$ will be imposed throughout the paper.

\begin{thm}[Lyapunov Spectrum of a Blaschke product cocycle]
\label{thm:LyapSpect}
Let $\sigma$ be an invertible ergodic measure-preserving transformation
of a probability space $(\Omega,\PP)$. Let $R<1$ and let
$\mathcal T=(T_\omega)_{\omega\in\Omega}$
be a family of finite Blaschke products, depending measurably on $\omega$,
satisfying $r:=r_{\mathcal T}(R)<R$.
Let $\LL_\omega$ denote the Perron-Frobenius operator of $T_\omega$,
acting on the Hilbert space $H^2(A_R)$ on the annulus 
$A_R\colon=\{z\colon R < |z|<1/R\}$.

Then the cocycle is compact and the following hold:
\begin{enumerate}
\item(Random Fixed Point).\label{it:randomfp}
There exists a measurable map $x\colon \Omega\to \overline {D_r}$ 
(with $x(\omega)$ written as $x_\omega$), such that 
$T_\omega(x_\omega)=x_{\sigma(\omega)}$.
For all $z\in D_{R}$, $T_{\sigma^{-N}\omega}^{(N)}(z):=
T_{\sigma^{-1}\omega}\circ\cdots
\circ T_{\sigma^{-N}\omega}(z)\to x_\omega$;
\item (Critical Random Fixed Point) If $\PP(\{\omega\colon 
T_\omega'(x_\omega)=0\})>0$, then the Lyapunov spectrum \label{it:criticalfp}
of the cocycle is 0 with multiplicity 1; and $-\infty$ with infinite multiplicity.
\item (Generic case) If $\PP(\{\omega\colon T_\omega'(x_\omega)=0\})=0$, then define
$\Lambda=\int\log |T_\omega'(x_\omega)|\,d\PP(\omega)$. This satisfies
$\Lambda\le \log(r/R)<0$.\label{it:generic}

If $\Lambda=-\infty$, then the Lyapunov spectrum of the cocycle is 0 with multiplicity 1;
and $-\infty$ with infinite multiplicity.

If $\Lambda>-\infty$, then the Lyapunov spectrum of the cocycle is
0 with multiplicity 1; and $n\Lambda$ with multiplicity 2 for each $n\in\N$.
The Oseledets space with exponent 0 is spanned by $1/(z-x_\omega)-1/(z-1/\bar x_\omega)$.
The Oseledets space with exponent $j\Lambda$ is spanned by two functions,
one a linear combination of 
$1/(z-x_\omega)^2,\ldots, 1/(z-x_\omega)^{j+1}$ with a pole of order $j+1$ at $x_\omega$;
the other a linear combination of $z^{k-1}/(1-\bar x_\omega z)^{k+1}$ for $k=1,\ldots,j$ with
a pole of order $j+1$ at $1/\bar x_\omega$. 
\end{enumerate}
\end{thm}

\begin{cor}\label{cor:unpert}
Let $\Omega=\{0,1\}^\Z$, $\sigma$ be the shift map,
and $\PP_p$ be the Bernoulli measure where $\PP([0])=p$.
Let $\LL_0$ be the Perron-Frobenius operator of $T_0\colon z\mapsto z^2$
and $\LL_1$ be the Perron-Frobenius operator of $T_1\colon z\mapsto 
[(z+\frac 14)/(1+\frac z4)]^2$
and consider the operator cocycle generated by $\LL_\omega:=\LL_{\omega_0}$, 
acting on $H^2(A_R)$, where $R$ satisfies $r_{\mathcal T}(R)<R$. Then
\begin{enumerate}[label=(\alph*)]
\item\label{it:finite}
If $p<\frac 12$, then the cocycle has countably infinitely many finite Lyapunov exponents;
\item \label{it:infinite}
If $p\ge \frac 12$, then $\lambda_1=0$ and all remaining 
Lyapunov exponents are $-\infty$.
\end{enumerate}
\end{cor}

We define an operator on $H^2(A_R)$ by
$$
(\mathcal N_\epsilon f)(z)=\frac{1}{\sqrt{2\pi}}\int_{-\infty}^\infty 
f(ze^{-2\pi i\epsilon t})e^{-2\pi i\epsilon t-t^2/2}\,dt.
$$
We show below that this operator corresponds to the operator on functions
on $\R/\Z$ given by $\mathcal N^{\R/\Z}_\epsilon f(x)=
\mathbb Ef(x+\epsilon N)$, where $N$ is a standard normal random variable. 
That is, $\mathcal N_\epsilon^{\R/\Z}$ acts on densities by convolution 
with a Gaussian with variance $\epsilon^2$. Let $\LL_\omega$ be the cocycle
in the theorem above and consider the perturbation $\LL^\epsilon_\omega$ of
$\LL_\omega$ given by $\LL^\epsilon_\omega=\mathcal N_\epsilon\circ \LL_\omega$.

\begin{cor}[Collapse of Lyapunov spectrum]\label{cor:pert}
Let $\Omega=\{0,1\}^\Z$, equipped with the map $\sigma$ and measure $\PP_p$
as above. If $p\ge \frac 14$, the perturbed cocycle 
$(\LL_\omega^\epsilon)_{\omega\in\Omega}$ has $\lambda_1=0$
and $\lambda_j=-\infty$ for all $j>1$. 

In particular, if $\frac 14\le p<\frac 12$, then the unperturbed cocycle has a complete
Lyapunov spectrum, but for each $\epsilon>0$, the Lyapunov spectrum collapses.
\end{cor}

Similarly, we define $\LL^{U,\epsilon}_i:=\mathcal U_\epsilon\circ\LL_i$, where 
$\mathcal U^{\R/\Z}_\epsilon$ convolves densities
with a bump function of integral 1 supported on $[-\epsilon, \epsilon]$,
$$
(\mathcal U_\epsilon f)(z)=\frac{1}{2}\int_{-1}^1 
f(ze^{-2\pi i\epsilon t})e^{-2\pi i \epsilon  t}\,dt.
$$
\begin{cor}\label{cor:Unifpert}
Let $\Omega=\{0,1\}^\Z$, equipped with the map $\sigma$ and measure $\PP_p$
as in Corollary~\ref{cor:unpert}. For every $p>0$, there are arbitrarily 
small values of $\epsilon$ for which the perturbed cocycle 
$(\LL_\omega^{U,\epsilon})_{\omega\in\Omega}$ has $\lambda_1=0$
and $\lambda_j=-\infty$ for all $j>1$. 
\end{cor}

The above corollaries give simple explicit examples and perturbations of Perron-Frobenius cocycles
for which the Lyapunov spectrum collapses. We now give general conditions for stability
and instability.

\begin{thm}[Stability of Lyapunov Spectrum]\label{thm:Lyapcont}
Let $\sigma$ be an ergodic invertible measure-preserving transformation of 
$(\Omega,\PP)$. Let $R<1$ and let $\mathcal T=(T_\omega)_{\omega\in\Omega}$
be a Blaschke product cocycle satisfying $r_{\mathcal T}(R)<R$.

\begin{enumerate}[label=(\alph*)]
\item \label{it:stabpert}
Suppose $\essinf_\omega |T_\omega'(x_\omega)|>0$.
Then if $(\LL_\omega)$ is the Perron-Frobenius cocycle of $(T_\omega)$
and $(\LL_\omega^\epsilon)$ is a family of Perron-Frobenius cocycles
such that $\esssup_\omega\|\LL_\omega^\epsilon-\LL_\omega\|\to 0$ as $\epsilon\to 0$,
then $\mu_k^\epsilon\to \mu_k$ as $\epsilon\to 0$, where $(\mu_k)$ is the sequence
of Lyapunov exponents of $(\LL_\omega)$, listed with multiplicity
and $(\mu_k^\epsilon)$ is the sequence of Lyapunov exponents of 
$(\LL_\omega^\epsilon)$.
\item \label{it:unstpert}
Suppose $\essinf_\omega|T_\omega'(x_\omega)|=0$.
Then there exists a family of Blaschke product cocycles $\mathcal T^\epsilon=
(T^\epsilon_\omega)_{\omega\in\Omega}$
such that $\esssup_\omega\|\LL_{T_\omega^\epsilon}-\LL_{T_\omega}\|\to 0$ as $\epsilon\to 0$
with the property
that the Lyapunov exponents of $(\LL_{T_\omega^\epsilon})$ are 0 with multiplicity 1, and $-\infty$
with infinite multiplicity for all $\epsilon>0$.
\end{enumerate}
\end{thm}

\begin{remark}
As a corollary of part \ref{it:stabpert}, it follows from \cite{FGQ-TAMS}
that the Oseledets spaces also converge in probability to those of the 
unperturbed cocycle. In fact, the cone arguments used here allow us
to draw the stronger conclusion that the Oseledets spaces converge 
essentially uniformly in $\omega$ to those of the unperturbed cocycle.
\end{remark}

\begin{remark}
Note that part \ref{it:stabpert} allows quite general perturbations to the Perron-Frobenius 
cocycle, while part \ref{it:unstpert} shows that if the random fixed
point has unbounded contraction, there are counterexamples even within the class of
Perron-Frobenius operators of Blaschke products.
For these counterexamples, the corresponding Blaschke products satisfy
$\esssup_\omega\max_{z\in \bar{D}_1}|T^\epsilon_\omega(z)-T_\omega(z)|<\epsilon$.
\end{remark}

Let $\Blaschke$ denote the set of measurable maps 
$\mathcal T:\omega\mapsto T_\omega$
from $\Omega$ to the collection of Blaschke products such that 
$r_{\mathcal T}(R)<R$. We equip $\Blaschke$ with the distance
$d(\mathcal T,\mathcal S)=\esssup_\omega\max_{z\in \bar{D}_1}
|T_\omega(z)-S_\omega(z)|$ (which, by the maximum principle, is the
same as $\esssup_\omega\max_{z\in C_1}|T_\omega(z)-S_\omega(z)|$).
We say that $\mathcal T\in\Blaschke$ is \emph{stable} if the conditions
of Theorem \ref{thm:Lyapcont}\ref{it:stabpert} apply.

\begin{cor}\label{cor:stabdense}
Let $\sigma$ be an invertible ergodic measure-preserving transformation
of $(\Omega,\PP)$. The stable Blaschke product cocycles form an open
dense subset of $\Blaschke$.
\end{cor}

The following theorem allows us to analyse the Lyapunov spectrum of Perron-Frobenius cocycles
of Blaschke products acting on coarser Banach spaces than $H^2(A_R)$.

\begin{thm}\label{thm:METcomp}
Let $\mathcal R=(\Omega,\PP,\sigma,X,\LL)$ be a random linear dynamical system
let $X'$ be a dense subspace of $X$ equipped with a finer norm such that $\LL_\omega(X')\subset X'$
for a.e.\ $\omega$. Let $\mathcal R'=(\Omega,\PP,\sigma,X',\LL|_{X'})$ be the restriction 
of $\mathcal R$ to $X'$. Suppose that $\mathcal R$ and $\mathcal R'$ both satisfy the 
conditions of Theorem \ref{thm:MET} (Multiplicative Ergodic Theorem). Then the exceptional Lyapunov
exponents of $\mathcal R$ and $\mathcal R'$ that exceed $\max(\kappa(\mathcal R),\kappa(\mathcal R'))$
coincide, as do the corresponding Oseledets spaces.
\end{thm}

\section{background}
Recall that a (finite) \emph{Blaschke product} is a map from $\hat\C$ to itself of the form:
$$
T(z)=\zeta\prod_{j=1}^n \frac{z-\zeta_j}{1-\bar \zeta_jz},
$$
where the $\zeta_j$'s lie in $D$, the open unit disc and $|\zeta|=1$. 
We say a function from $\Omega$ into the collection of Blaschke products is measurable
if each of the parameters, $n$, $\zeta$, $\zeta_1,\ldots,\zeta_n$ is a Borel measurable function 
of $\omega$. We note also that if $S(z)=\xi\prod_{j=1}^n (z-\xi_j)(1-\bar \xi_jz)^{-1}$ and
$(\xi, \xi_1, \dots, \xi_n)$ is sufficiently close to $(\zeta, \zeta_1, \dots, \zeta_n)$, 
then $\max_{z\in C_1}|S(z)-T(z)|<\epsilon$.

We record the following:
\begin{lem}[Properties of finite Blaschke products]\label{lem:Blaschkprop}
Let $T$ be a Blaschke product. Then
\begin{enumerate}[label=(\alph*)]
\item $T(C_1)=C_1$;
\item $T\circ I=I\circ T$, where $I$ is the inversion map $I(z)=1/\bar z$;
\item $T$ maps $D_1$ to itself (and hence $T$ maps $\hat \C\setminus \bar D_1$ to itself);
\end{enumerate}
Further,
\begin{enumerate}[resume, label=(\alph*)]
\item if $T$ is a non-constant map from the closed unit disc to itself that is
analytic in the interior and maps the boundary to itself, then $T$ is a finite Blaschke product.
\label{it:Blaschkchar}
\item
The composition of two finite Blaschke products is again a finite Blaschke product.
\label{it:Blaschkcomp}
\end{enumerate}
\end{lem}

We are interested in finite Blaschke products whose restriction to the unit circle is
expanding. A simple sufficient condition for this, namely that 
$\sum_{j=1}^n \frac{1-|\zeta_j|}{1+|\zeta_j|}>1$,  may be found in the 
work of Martin \cite{Martin}.

Let $\pi(x)=e^{2\pi ix}$ be the natural bijection between $\R/\Z$ 
and $C_1$, the unit circle in the complex plane.
Let $S$ be an orientation-preserving expanding real analytic map 
from $\R/\Z$ to $\R/\Z$ and let 
$T=\pi S\pi^{-1}$ be its conjugate mapping from $C_1$ to $C_1$. Then there
exist $r<R<1$ such that $T$ maps the annulus $A_R=\{z\colon R<|z|<1/R\}$
over $A_r$ in a $k$-to-1 way (where $k$ is the absolute value of the degree of $S$). 
We will work with a family of expanding analytic maps of $C_1$, all mapping
$C_R$ into $\bar D_r$ (and hence $C_{1/R}$ into $D_{1/r}^c$).
We consider the separable Hardy-Hilbert space $H^2(A_R)$ of analytic functions on $A_R$ 
with an $L^2$ extension to $\partial A_R$. An orthonormal basis for the Hilbert space 
is $\{d_nz^n\colon n\in\Z\}$, where $d_n=R^{|n|}(1+R^{4|n|})^{-1/2}=R^{|n|}(1+o(1))$.
More details can found in \cite{BJS}.

Let $(\sigma_i)_{i=1,\ldots,k}$ be a family of inverse branches of $S$ and $(\tau_i)_{i=1,\ldots,k}$ 
be a family of inverse branches of $T\vert_{C_1}$.

If $f\in C(C_1)$ and $g\in C(\R/\Z)$, define
\begin{align*}
\LL_Tf(z)&=\sum_{i=1}^k f(\tau_i(z))\tau_i'(z)
\text{; and}\\
\LL_Sg(x)&=\sum_{i=1}^k g(\sigma_i(x))\sigma_i'(x)
\end{align*}

In fact, rather than acting on $C(C_1)$, we think of $\LL_T$ as acting on $H^2(A_R)$.
This corresponds to an action of $\LL_S$ on the strip 
$\R/\Z\times(\log R/(2\pi),-\log R/(2\pi))$.
Note that there is a one--one correspondence between elements of $H^2(A_R)$
and their restrictions to $C_1$ (which are necessarily continuous). 
\begin{figure}
\includegraphics[width=2in]{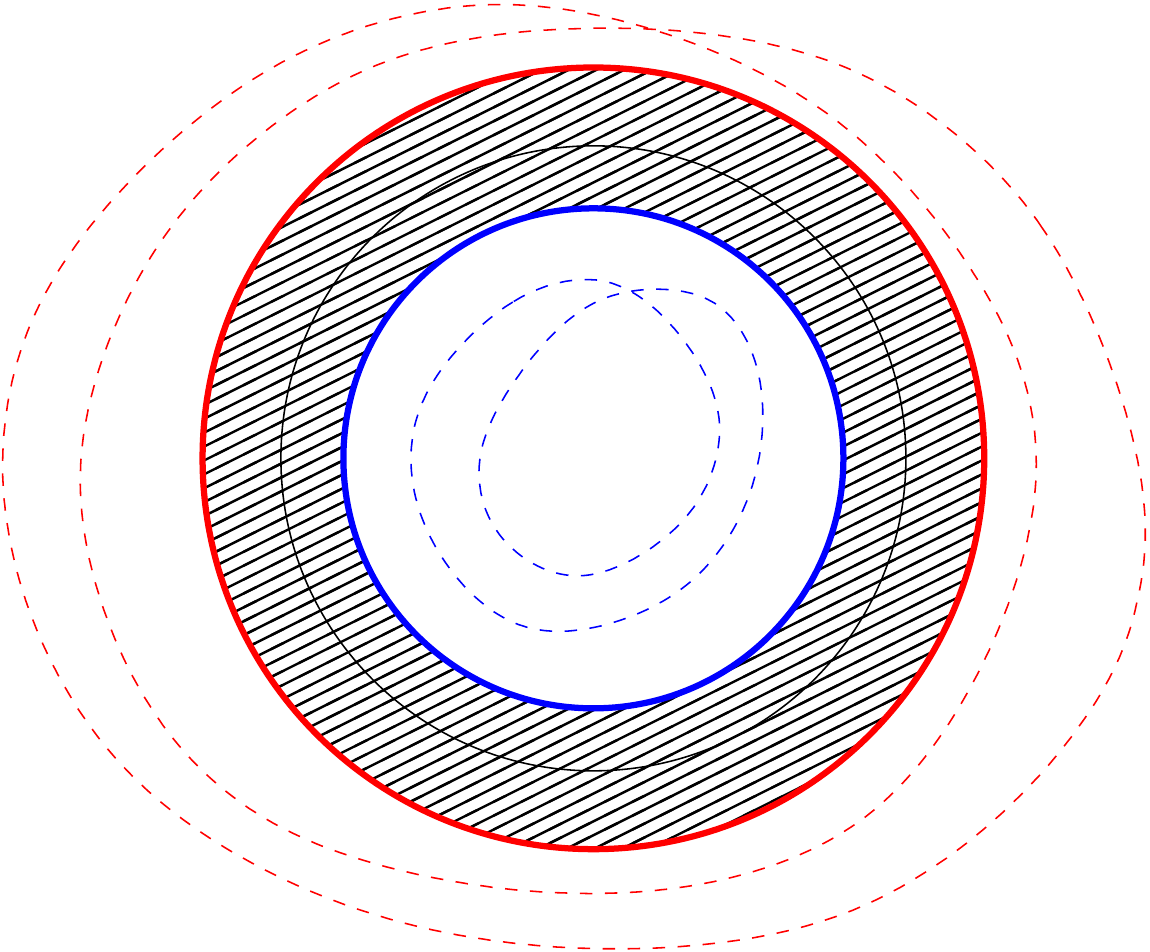}
\caption{Schematic diagram showing the annulus $A_R$ (shaded); the 
inner boundary, $C_R$ (blue); its image under $T$ (dashed, blue);
the outer boundary $C_{1/R}$ (red); and its image (dashed, red).}
\end{figure}

We record the following lemma which is a version of \cite[Remark 4.1]{SBJ}
\begin{lem}[Correspondence between Perron-Frobenius Operators]\label{lem:corr}
The operators $\LL_S$ and $\LL_T$ defined above are conjugate
by the map $\mathcal Q\colon C(C_1)\to C(\R/\Z)$ given by 
$$
(\mathcal Qf)(x)=f(e^{2\pi ix})e^{2\pi ix}.
$$
\end{lem}

In particular, the spectral properties of $\LL_S$ and $\LL_T$ are the 
same, so that even though $\LL_S$ is the more standard object in 
dynamical systems, we will study $\LL_T$ since this will allow us to 
directly apply tools of complex analysis. Further, if $S_{i_1},\ldots,S_{i_n}$ are
expanding maps of  $\R/\Z$ and $T_{i_1},\ldots,T_{i_n}$ are their conjugates, acting
on $C_1$, then $\LL_{S_{i_n}}\circ\cdots \circ\LL_{S_{i_1}}=
\mathcal Q^{-1}\LL_{T_{i_n}}\circ\cdots\circ
\LL_{T_{i_1}}\mathcal Q$, so the Lyapunov exponents of a cocycle of 
$\LL_S$ operators are
the same as the Lyapunov exponents of the corresponding cocycle of $\LL_T$ operators
(provided that $\mathcal Q$ is an isomorphism of the two Banach 
spaces on which the operators are acting). 

We record the following well known lemma.

\begin{lem}[Duality Relations]\label{lem:dual}
Let $T$ be an expanding analytic map from $C_1$ to $C_1$ and let 
$\mathcal L_T$ be as above. 
If $f\in C(C_1)$, $g\in L^\infty(C_1)$, then
$$
\frac{1}{2\pi i}\int_{C_1} f(z)g(Tz)\,dz=
\frac{1}{2\pi i}\int_{C_1} \mathcal L_Tf(z)g(z)\,dz.
$$
\end{lem}
That is, if a linear functional $\theta$ is defined by integrating against $g$, 
then $\mathcal L_T^*\theta$
is given by integrating against $g\circ T$.

Let $\sigma$ be an ergodic measure-preserving transformation of $(\Omega,\PP)$.
Let $X$ be a Banach space and suppose that $(\LL_\omega)_{\omega\in\Omega}$ is a
family of linear operators on $X$ that is \emph{strongly measurable}, that is
for any fixed $x\in X$, $\omega\mapsto\LL_\omega(x)$ is $(\mathcal F_\Omega,
\mathcal F_X)$-measurable, where $\mathcal F_\Omega$ is the $\sigma$-algebra
on $\Omega$ and $\mathcal F_X$ is the Borel $\sigma$-algebra on $X$.
In this case we say that the tuple $\mathcal R=(\Omega,\PP,\sigma,X,\mathcal L)$ is a
\emph{random linear dynamical system} and we define $\mathcal L^{(n)}_\omega=
\LL_{\sigma^{n-1}\omega}\circ\cdots\circ\LL_\omega$.

A cocycle analogue to the (logarithm of the) 
spectral radius of a single operator is the quantity
$$
\lambda_1(\omega)=\lim_{n\to\infty}\tfrac 1n\log\|\LL^{(n)}_\omega\|.
$$
If one assumes that $\int\log\|\LL_\omega\|\,d\PP<\infty$, then the
Kingman sub-additive ergodic theorem guarantees the 
$\PP$-a.e.\ convergence of this limit to a value in $[-\infty,\infty)$. 
Ergodicity also ensures that $\lambda_1(\cdot)$
is almost everywhere constant, so that we just write $\lambda_1$. 

A second quantity of interest is the analogue of the (logarithmic) essential spectral
radius, the \emph{asymptotic index of compactness} \cite{Thieullen}:
$$
\kappa(\omega)=\lim_{n\to\infty}\tfrac 1n\log\alpha(\LL^{(n)}_\omega),
$$
where $\alpha(L)$ is the \emph{index of compactness} of an operator $L$, the
infimum of those real numbers $t$ such that the 
image of the unit ball in $X$ under $L$ may be covered by finitely many balls 
of radius $t$, so that $L$ is a compact operator if and only if $\alpha(L)=0$.
The quantity $\alpha(L)$ is also sub-multiplicative, so that Kingman's
theorem again implies $\kappa(\omega)$ exists for $\PP$-a.e. $\omega$ and is
independent of $\omega$, so that we just write $\kappa$. The cocycle will be called
\emph{quasi-compact} if $\kappa<\lambda_1$. The first Multiplicative Ergodic Theorem 
in the context of quasi-compact cocycles of operators on Banach spaces was proved
by Thieullen \cite{Thieullen}. We require a semi-invertible
version (that is: although the base dynamical system is required to be invertible,
the operators are not required to be injective) of a result of Lian and Lu \cite{LianLu}.

\begin{thm}[\cite{GTQ-JMD}]\label{thm:MET}
Let $\sigma$ be an invertible ergodic measure-preserving transformation of 
a probability space $(\Omega,\PP)$ and
let $\omega\mapsto\LL_\omega$ be a quasi-compact 
strongly measurable cocycle of operators acting on a
Banach space $X$ with a separable dual satisfying $\int\log\|\LL_\omega\|\,d\PP(\omega)<\infty$.

Then there exist $1\le\ell\le\infty$, 
exponents $\lambda_1\ge \lambda_2\ge \ldots\ge\lambda_\ell\ge\kappa\ge
-\infty$, finite multiplicities $m_1,m_2,\ldots,m_\ell$ and subspaces 
$V_1(\omega),\ldots,V_\ell(\omega),W(\omega)$ such that 
\begin{enumerate}[label=(\alph*)]
\item $\dimn(V_i(\omega))=m_i$; 
\item $\LL_\omega V_i(\omega)=V_i(\sigma(\omega))$ and
$\LL_\omega W(\omega)\subset W(\sigma(\omega))$;
\item $V_1(\omega)\oplus\ldots\oplus V_\ell(\omega)\oplus W(\omega)=X$;
\item for $x\in V_i(\omega)\setminus\{0\}$, $\lim\frac 1n\log
\|\mathcal L^{(n)}_\omega x\|\to\lambda_i$;
\item for $x\in W(\omega)$, $\limsup\frac 1n\log \|\LL^{(n)}_\omega x\|
\le \kappa$.
\end{enumerate}
\end{thm}

For a bounded linear operator $A$ from $X$ to itself, we defined 
the following crude notion of volume growth in \cite{GTQ-JMD}:
$$
\mathcal D_k(A)=
\sup_{x_1,\ldots,x_k}\prod_{j=1}^kd(A x_j,\text{lin}(\{Ax_i\colon i<j\})),
$$
where the supremum is taken over $x$'s of norm 1; $\lin(\{y_1,\ldots,y_n\})$ denotes
the linear span of the vectors $y_1,\ldots,y_n$; the linear span of the empty set
is taken to be \{0\}; and $d(x,S):=\inf_{y\in S}\|x-y\|$.

\begin{lem}\label{lem:exptsviavol}
Let $\sigma$, $(\Omega,\PP)$ and $\omega\mapsto\LL_\omega$ 
be as in the statement of 
Theorem \ref{thm:MET}. Let $\mu_1\ge \mu_2\ge \ldots$ be
the sequence of $\lambda$'s in decreasing order \emph{with repetition}, so that 
$\lambda_i$ occurs $m_i$ times in the sequence.

\begin{enumerate}[label=(\alph*)]
\item $\mathcal D_k$ is sub-multiplicative: $\mathcal D_k(AB)\le 
\mathcal D_k(A)\mathcal D_k(B)$ if $A$ and $B$
are bounded linear operators on $X$;\label{it:Dsubmult}
\item There exists a constant $c_k$ such that 
if $Y$ is a closed subspace of $X$ of co-dimension 1 and $A$ is a linear operator on $X$,
then $\mathcal D_k(A)\le c_k\|A\|\|A|_Y\|^{k-1}$.\label{it:Dfactor}
\item 
$\frac 1n\log \mathcal D_k(\LL_\omega^{(n)})\to \mu_1+\ldots+\mu_k$ for
$\PP$-almost every $\omega$;\label{it:Dexpts}
\end{enumerate}
\end{lem}
The proof of this is in Lemmas 1,8 and 12 of \cite{GTQ-JMD}.

\section{Lyapunov Spectrum for expanding Blaschke products}
\label{sec:LyapSpectrum}

\begin{lem}\label{lem:contraction}
Let $R<1$ and let $T$ be a Blaschke product satisfying $r:=r_T(R)<R$.
Let $d_R$ be the hyperbolic metric on $D_{R}$: $d_R(z,w)=d_H(z/R,w/R)$,
where $d_H$ is the standard hyperbolic metric on the unit disc. Then
$$
d_R(T(z),T(w))\le \tfrac rR \,d_R(z,w)
\text{ for all $z,w\in D_R$.}
$$
\end{lem}

\begin{proof}
We may write $T$ as $Q\circ S$ where $Q(z)=rz/R$ and $S(z)=RT(z)/r$,
so that $S$ maps $D_R$ to itself. By the Schwartz-Pick theorem, 
$d_R(S(z),S(w))\le d_R(z,w)$ for all $z,w\in D_R$, 
so it suffices to show that $d_R(Q(z),Q(w))\le
\frac rR\,d_R(z,w)$ for all $z,w\in D_R$. 
The metric $d_R$ is given, up to a constant multiple by
$$
d_R(z,w)=\inf_\gamma \int \frac{|d\xi|}{1-|\xi|^2/R^2}\,
$$
where the infimum is taken over paths $\gamma$ from $z$ to $w$. Given $z$ and $w$,
let $\gamma$ be the geodesic joining them. Now $(r/R)\gamma(t)$ is a path (generally not 
a geodesic) joining $Q(z)$
and $Q(w)$. The length element is scaled by a factor of $r/R$ and the integrand is
decreased,
so that $d_R(Q(z),Q(w))\le \frac rR\,d_R(z,w)$ as claimed.
\end{proof}

\begin{cor}\label{cor:Blaschkefp}
Let $R<1$ and $\mathcal T=(T_\omega)_{\omega\in\Omega}$ be a measurable
cocycle of expanding finite Blaschke products satisfying $r:=r_{\mathcal T}(R)<R$.
Then there exists a measurable
random fixed point $x_\omega$ (that is a point such that $T_\omega(x_\omega)=
x_{\sigma(\omega)}$)
in $\overline{D_r}$ such that for all $\epsilon>0$, there exists $n$ such that for all 
$z\in \overline{D_R}$ and a.e.\ $\omega\in\Omega$, 
$|T_{\sigma^{-n}\omega}^{(n)}(z)-x_\omega|<\epsilon$.
\end{cor}

\begin{proof}
The set $\overline{D_r}$ has bounded diameter, $L$ say, 
in the $d_R$ metric and by assumption, for a.e.\ $\omega$,
the sets $T_{\sigma^{-n}\omega}^{(n)}(\overline{D_R})$ are nested.
By Lemma \ref{lem:contraction}, $T_{\sigma^{-n}\omega}^{(n)}(\overline{D_R})$
has $d_R$-diameter at most $L(\frac rR)^{n-1}$. 
By completeness,
$\bigcap T_{\sigma^{-n}\omega}^{(n)}(\overline{D_R})$ is a 
singleton, $\{x_\omega\}$.
Since $x_\omega=\lim_{n\to\infty}T_{\sigma^{-n}\omega}^{(n)}(0)$, and so is the limit 
of measurable functions, we see that $x_\omega$ depends measurably on $\omega$.
This equality also implies that $x_{\sigma(\omega)}=T_\omega(x_\omega)$. 
Since on $D_r$, $d_R$ is within a bounded factor of Euclidean distance, 
we obtain the required uniform convergence in the Euclidean distance. 
\end{proof}

We introduce a non-standard definition of order of singularity for rational
functions on the Riemann sphere: if $x\in\C$, $\ord_x(f)$ is $n$ if $f(z)\sim
a/(z-x)^n$ as $z\to x$ for some $n\ge 1$; or 0 otherwise. 
If $f(z)\sim bz^{n-2}$ for some $n\ge 1$ as $z\to\infty$ then 
$\ord_\infty(f)=n$ or $\ord_\infty(f)=0$ otherwise. 
In particular, with this definition $\ord_\infty(1)=2$. If $f$ is a non-zero rational 
function on the Riemann sphere, then $\sum_{z\in\hat\C}\ord_z(f)\ge 2$: if $f$ is a 
non-zero constant function, it has a singularity of order 2 at $\infty$; otherwise 
$f$ must have at least one pole by Liouville's theorem. If $f$ has exactly one pole,
of order 1 at $x$ say, then $f-a/(z-x)$ (where $a$ is the residue) 
is bounded and hence constant, but $f(z)=c+a/(z-x)$ has order 1 at $\infty$ if $c=0$
or order 2 at $\infty$ otherwise.

\begin{lem}\label{lem:ordfacts}
If $Q$ is a M\"obius transformation, then $\LL_Q$ maps the collection of rational functions on
$\hat\C$ into itself. Further, $\ord_{Q(x)}(\LL_Q f)=\ord_x f$ for all $x\in\hat \C$.
\end{lem}

\begin{proof}
It suffices to check that $\LL_Qf$ is meromorphic on a neighbourhood of all points in $\hat \C$.
If $Q(w)=z$ with $w,z\ne\infty$, then it is clear that $\LL_Qf$ is meromorphic near $z$
and $\ord_{\LL_Qf}(z)=\ord_{f}(w)$. 
If $Q(z)=1/z$, then a calculation shows $\LL_Qf(z)=-f(\frac 1z)/z^2$ and we
can check that $\ord_0(\LL_Qf)=\ord_\infty f$ and $\ord_\infty(\LL_Qf)=\ord_0 f$.
Since $\LL_{S\circ T}=\LL_S\circ \LL_T$ and every M\"obius transformation can be
expressed as a composition of maps of the form $z\mapsto az+b$ with $a\ne 0$ and 
$z\mapsto 1/z$, the proof is complete.
\end{proof}

\begin{thm}\label{thm:meroPF}
Let $T$ be a rational map. Then $\LL_T$ maps the collection of rational functions on 
$\hat \C$ into itself. $\LL_T$ does not increase orders of singularities and may decrease them:
$\ord_x \LL_Tf\le \max_{y\in T^{-1}x}\ord_y f$ for each $x\in\hat \C$. 

Further, if $T$ has a critical point at $x$, $f$ has a singularity of order greater than 1 at $x$ and
no singularity at any other point of $T^{-1}(Tx)$, then
$\ord_{T(x)}\LL_T f<\ord_x f$. 
\end{thm}

The inequality in the first paragraph is an equality except at
points $x$ such that $T^{-1}(Tx)$ contains a critical point at which $f$ has a singularity,
or contains multiple singularities of $f$. We remark that the theorem seems quite surprising 
since the inverse branches of rational functions are not, typically rational.

\begin{proof}
By Lemma \ref{lem:ordfacts}, it suffices, by pre- and post-composing $T$ with
M\"obius transformations if necessary, to consider the case $z\ne\infty$ and $T(\infty)\ne z$. 
First, notice that if $T$ has no critical points or 
poles in $T^{-1}z$, then it is clear from the definition of 
$\LL_T$ that $\LL_Tf$ is analytic in a neighbourhood of $z$. In particular, 
$\LL_Tf$ is analytic off the finite set of images of critical points of $T$ and poles of $f$. 

Now let $T^{-1}z=\{y_1,\ldots,y_k\}$ and let $T(z)-T(y_i)$ have a zero of order $m_i\ge 1$ for 
$i=1,\ldots,k$. Suppose further that $\ord_{y_i}f=o_i$. Let $\delta$ be sufficiently small that
$T^{-1}B_\delta(z)$ consists of $k$ disjoint neighbourhoods $N_1,\ldots,N_k$ of $y_1,\ldots,y_k$. 
Now for $0<|h|<\delta$, we have 
$$
\LL_Tf(z+h)=
\sum_{j=1}^k \sum_{y\in T^{-1}(z+h)\cap N_j}\frac{f(y)}{T'(y)}.
$$
If $y\in T^{-1}(z+h)\cap N_j$, we have $|y-y_j|=O(h^{1/m_j})$ and 
$|1/T'(y)|=O(h^{-1+1/m_j})$ and $|f(y)|=O(h^{-o_j/m_j})$.
Hence $\LL_Tf(z+h)=O(h^{-1-\max_j (o_j-1)/m_j})$.
This guarantees that $\LL_Tf$ does not have an essential singularity at $z$,
so that $\LL_Tf$ is meromorphic on a neighbourhood of $z$ as claimed. 
Further, since meromorphic functions cannot exhibit fractional power growth rates,
we have 
$\ord_z\LL_Tf\le 1+\max_j \left\lfloor\frac{o_j-1}{m_j}\right\rfloor\le\max o_j$ as required.
\end{proof}

\begin{cor}\label{cor:uppertri}
Let $T$ be a rational function and let $x\in\C$ satisfy $T(x)\in\C$. 
If $f(z)=1/(z-x)^{n+1}$ for some $n\ge1$, then $\LL_T$ is a linear 
combination of $\{1/(z-T(x))^{j+1}\colon 1\le j\le n\}$. 
\end{cor}

\begin{proof}
Since $f$ has only one singular point, at $x$, of order $n+1$, 
then $\LL_Tf$ is a rational function on the sphere with only
one singular point, at $T(x)$, of order at most $n+1$. In particular, $\LL_Tf=O(z^{-2})$
in a neighbourhood of $\infty$, which ensures that there is no 
$1/(z-T(x))$ term in $\LL_Tf$.
\end{proof}

We now introduce an operator that commutes with the Perron-Frobenius operators of 
Blaschke products, performing inversion at the level of meromorphic functions. 
This will allow us to focus on poles inside the unit disc, and avoid dealing separately
with poles at $\infty$. A precursor appears in \cite[Lemma 3.1c]{SBJ}. Define 
$$
\LL_If(z)=\frac{\bar f(I(z))}{z^2}.
$$

\begin{lem}
The operator $\LL_I$ has the following properties:
\begin{enumerate}[label=(\alph*)]
\item If $T$ be a finite Blaschke product,$\LL_T\LL_I=\LL_I\LL_T$;\label{it:commute}
\item $\LL_I$ maps meromorphic functions to meromorphic functions.\label{it:refl}
\item $\LL_I$ is a bounded anti-linear involution.\label{it:bddinv}
\item $\ord_{I(z)}\LL_If=\ord_z f$ for $f$ meromorphic and $z\in\hat\C$.\label{it:ordpres}
\end{enumerate}
\end{lem}

\begin{proof}
We first show \ref{it:commute}.
Using the identity $I\circ T=T\circ I$,
we have
\begin{align*}
T'(z)=\lim_{h\to 0}\frac{T(z+h)-T(z)}h&=
\lim_{h\to 0}\frac{\frac1{\overline{T}(1/(\bar z+\bar h))}
-\frac 1{\overline{T}(1/\bar z)}}h\\
&=\lim_{h\to 0}\frac{{\bar T(1/\bar z)- \bar T(1/(\bar z+\bar h))}}
{h\bar T(1/\bar z)^2}\\
&=\frac{\overline{ T'(1/\bar z)}}{z^2\bar T(1/\bar z)^2}.
\end{align*}
We deduce $\overline{T'(I(y))}=y^2T'(y)/T(y)^2$.
Now we have
$$
\LL_T\LL_I f(z)=\sum_{y\in T^{-1}z}\frac{\LL_If(y)}{T'(y)}
=\sum_{y\in T^{-1}z}\frac{\bar f(I(y))}{y^2T'(y)}; 
$$
and
\begin{align*}
\LL_I\LL_T f(z)&=\overline{\LL_T f(I(z))}/z^2\\
&=\frac{1}{z^2}\sum_{y\in T^{-1}(I(z))}\frac{\bar f(y)}{\overline{T'(y)}}\\
&=\frac{1}{z^2}\sum_{y\in T^{-1}(z)}{\frac {\bar f(I(y))}{\overline{T'(I(y))}}}\\
&=\frac{1}{z^2}\sum_{y\in T^{-1}(z)}
\frac {\bar f(I(y)) T(y)^2}{ y^2 T'(y)}=\LL_T\LL_If(z).
\end{align*}
Part \ref{it:refl} is standard; the boundedness follows from the fact that $|\frac 1{z^2}|
\le \frac1{R^2}$ on $\partial A_R$; that $\LL_I$ is an involution and that it preserves 
orders are simple computations.
\end{proof}

\begin{lem}\label{lem:simplepoleit}
Let $T$ be a finite Blaschke product, let $x\in\C\setminus C_1$ and let $f(z)=1/(z-x)$. 
Then 
$$
\LL_T f(z)=\frac 1{z-T(x)}-\frac1{z-T(\infty)},
$$
where $1/(z-\infty)$ is interpreted as the constant 0 function.
\end{lem}

\begin{proof}
It follows from Theorem \ref{thm:meroPF} that $\LL_Tf$ is a rational function
with singularities of order 1 at $T(x)$ and $T(\infty)$:
$$
\LL_Tf=
\frac{a}{z-T(x)}+\frac b{z-T(\infty)}.
$$
In order that $\LL_Tf$ have no singularity
at $\infty$, we must have $b=-a$. 

If $|x|<1$, we see from Lemma \ref{lem:dual}
(taking the function $g$ to be 1) that $a=1$, as required.

If $|x|>1$, then $\LL_Tf=\LL_I\LL_T\LL_If$. We calculate $\LL_If(z)=1/z-1/(z-I(x))$, 
and using the first part
$\LL_T\LL_If(z)=1/(z-T(0))-1/(z-T(I(x)))$. Applying $\LL_I$ again, we arrive at
\begin{align*}
\LL_Tf(z)&=
\frac{1}{z-I(T(I(x)))}-\frac{1}{z-I(T(0))}\\
&=\frac{1}{z-T(x)}-\frac{1}{z-T(\infty)},
\end{align*}
as required.
\end{proof}

%

\begin{cor}\label{cor:topspace}
Let $\mathcal T=(T_\omega)$ be a cocycle of expanding finite Blaschke products.
Let $x_\omega\in D_{r}$ be the random fixed point. 
The space $E_0(\omega)$, spanned by $\hat e_{0,\omega}$ where 
$$
\hat e_{0,\omega}(z)=\frac1{z-x_\omega}-\frac1{z-I(x_\omega)}
$$
is a one-dimensional equivariant subspace with Lyapunov exponent 0.

If $x\in D_{r}$ and $f(z)=1/(z-x)$, then $\|\LL_\omega^{(n)}f
-\hat e_{0,\sigma^n\omega}\|\to 0$. 
\end{cor}

This corollary may be seen as a random version of a result
of Martin \cite{Martin}, expressing the invariant measure
of an expanding Blaschke product as a Poisson kernel.

\begin{proof}
By Lemma \ref{lem:simplepoleit} and the fact that $\LL_{T_1\circ T_2}=
\LL_{T_1}\circ\LL_{T_2}$, we see 
\begin{equation*}
\LL_\omega^{(n)}f(z)=\frac 1{z-T_\omega^{(n)}(x)}-
\frac1{z-T_\omega^{(n)}(\infty)}.
\end{equation*}
The claimed equivariance follows. 
Since $|T_{\omega}^{(n)}(x)-x_{\sigma^n\omega}|\to 0$,
we see that 
$$
\left\|\frac1{z-T_\omega^{(n)}(x)}-\frac1{z-x_{\sigma^n\omega}}\right\|\to 0.
$$
Similarly, since $T_\omega^{(n)}(\infty)=I(T_\omega^{(n)}(0))$, we see from
Corollary \ref{cor:Blaschkefp} that $d_{\hat \C}
(T_\omega^{(n)}(\infty),I(x_{\sigma^n\omega}))\to 0$, where
$d_{\hat \C}$ is the standard metric on the Riemann sphere.
It follows that 
$$
\left\|\frac1{z-T_\omega^{(n)}(\infty)}-
\frac1{z-I(x_{\sigma^n\omega})}\right\|\to 0,
$$
as required.
\end{proof}

We now show that the hypotheses of Theorem \ref{thm:MET} are satisfied
by a cocycle of Perron-Frobenius operators of Blaschke products satisfying 
$r_{\mathcal T}(R)<R$.

\begin{lem}\label{lem:PFcont}
Let $0<R<1$. If $r_T(R), r_{\tilde T}(R)\leq r < R$, then 
$\|\LL_T-\LL_{\tilde T}\|\le C\max_{x\in C_1}|T(x)-\tilde T(x)|$,
where $C$ is a constant depending only on $r$ and $R$.
In particular, restricted to Blaschke products satisfying $r_T(R)<R$,
the map $T\mapsto \LL_T$ is continuous.
\end{lem}

\begin{proof}
Recall that $H^2(A_R)$ is a Hilbert space with respect to the inner product
$\langle f,g\rangle=
\frac{1}{2\pi }\left(\int_{\partial A_R}f(z)\overline{g(z)}\,
\frac {|dz|}{|z|}\right).$
With respect to this inner product, the functions $e_n(z)=d_nz^n$ form an
orthonormal basis of $H^2(A_R)$, where $d_n=(R^{2n}+R^{-2n})^{-1/2}$,
so that $d_n\sim R^{|n|}$.

We now compute
\begin{align*}
\langle \LL_T(f),e_n\rangle&=
\frac{d_n}{2\pi i}\int_{\partial A_R}\LL_Tf(z)\bar z^n\,dz/z\\
&=\frac{d_n}{2\pi i}\left(
\int_{C_R}\LL_Tf(z)R^{2n}/z^n\,\frac{dz}z+
\int_{C_{1/R}}\LL_Tf(z)R^{-2n}/z^n\,\frac{dz}z\right)\\
&=\frac{d_n(R^{2n}+R^{-2n})}{2\pi i}\int_{C_1}\frac{\LL_Tf(z)}{z^{n+1}}\,dz\\
&=\frac{1}{2\pi id_n}\int_{C_1}\frac{f(z)}{T(z)^{n+1}}\,dz.
\end{align*}
Let $f$ be an arbitrary element  of $H^2(A_R)$ and let $n\ge 0$. 
Let $T$ and $\tilde T$ be any two Blaschke products satisfying
$r_T(R)\le r$ and $r_{\tilde T}(R)\le r$ for some $r<R$, and let $\delta=\max_{z\in C_1}|T(z)-\tilde T(z)|$.
Note by the maximum modulus principle, $|T(z)-\tilde T(z)|\le\delta$
for all $z\in D_1$ also.
Then deforming the contour to $C_{1/R}$, we see
\begin{align*}
|\langle (\LL_T-\LL_{\tilde T})f,e_n\rangle|
&\le
\frac{1}{2\pi d_n}\int_{C_{1/R}}|f(z)|\left|\frac{1}{T(z)^{n+1}}-
\frac{1}{\tilde T(z)^{n+1}}\right|\,|dz|\\
&\le \frac{1}{Rd_n}\|f\|\max_{z\in C_{1/R}}\left|\frac{1}{T(z)^{n+1}}-
\frac1{\tilde T(z)^{n+1}}\right|\\
&=\frac{1}{Rd_n}\|f\|\max_{z\in C_R}|T(z)^{n+1}-\tilde T(z)^{n+1}|\\
&\le \frac{(n+1)r^n}{Rd_n}\|f\|\max_{z\in C_R}|T(z)-\tilde T(z)|\\
&\le \tfrac{2(n+1)}{R}(\tfrac rR)^n\delta\|f\|,
\end{align*}
where for the third line, we used the fact that Blaschke products commute with inversion,
and for the fourth line, we used $r_T(R),r_{\tilde T}(R)\le r$.
If $n=-k$ with $k\ge 1$, then an analogous computation, deforming the contour
to $C_{R}$, shows
$$
|\langle (\LL_T-\LL_{\tilde T})f,e_n\rangle|
\le \tfrac{2(k-1)}{R}(\tfrac rR)^{k-2}\delta\|f\|.
$$
In particular, since the $(e_n)$ form an orthonormal basis, we deduce
$$
\|(\LL_T-\LL_{\tilde T})f\|\le C\delta\|f\|,
$$
where $C$ depends only on $r$ and $R$, as required.

We note that $r_T(R)$ depends continuously on $T$. Hence, if $r_T(R)<R$ and $\tilde T$ is sufficiently close to $T$,  then $r_{\tilde{T}}(R)\leq \frac{R+r_T(R)}{2}<R$ and the last statement of the lemma follows from the argument above.
\end{proof}

\begin{cor}\label{cor:compact}
Let $r<\rho<R<1$. There exists a $C>0$ such that
if the Blaschke product $T$ satisfies $r_T(R)\le r$,
then $\|\LL_T\|_{H^2(A_R)\to H^2(A_\rho)}\le C$.
In particular, $\LL_T$ is compact as an operator from $H^2(A_R)$ to itself.
\end{cor}

\begin{proof}
First, notice by the proof of Theorem \ref{thm:meroPF}
that $\LL_Tf$ is analytic on $A_r$.
As in the above proof, $\tilde e_n(z)=\tilde d_nz^n$ is
an orthonormal basis for $H^2(A_\rho)$, where $\tilde d_n=(\rho^{2n}+
\rho^{-2n})^{-1/2}$. As above 
$\langle \LL_Tf,\tilde e_n\rangle_{H^2(A_\rho)}=\frac{1}{2\pi i\tilde d_n}
\int_{C_1}f(z)/T(z)^{n+1}\,dz$. Deforming the contour to $C_{1/R}$ in the case where
$n\ge 0$ and $C_{R}$ when $n<0$, we obtain 
$|\langle \LL_Tf,\tilde e_n\rangle_{H^2(A_\rho)}|\le C(r/\rho)^{|n|}\|f\|$.
Since this is square summable, the result follows.
\end{proof}

In the context of Theorem \ref{thm:LyapSpect}, the map $\omega\mapsto T_\omega$
is measurable, and the map $T\mapsto \LL_T$ is continuous, so that the 
composition, $\omega\mapsto \LL_{T_\omega}$
satisfies the hypotheses of Theorem \ref{thm:MET}.

\begin{lem}\label{lem:fastapprox2}
Let $\mathcal R$ be a random linear dynamical system satisfying the conditions of 
Theorem \ref{thm:MET}.
Let $E_j(\omega)$ be the $j$th ``fast
space'' $V_1(\omega)\oplus \dots \oplus V_j(\omega)$ and let $F_j(\omega)$ be the complementary ``slow space''. 
If $V$ is a subspace of $X$ satisfying $\Pi_{E_j(\omega)\parallel F_j(\omega)}(V)=
E_j(\omega)$, then 
$$
\sup_{x\in E_j(\sigma^n\omega)\cap S(X)}
d(x,\LL_\omega^{(n)}V)\to 0\text{ as $n\to\infty$.}
$$
\end{lem}

\begin{proof}
We write $E$ and $F$  for $E_j(\omega)$
and $F_j(\omega)$.
Let $W$ be a subspace of $V$ of the same dimension as $E$ such that
$\Pi_{E\parallel F}(W)=E$. Let $Q=(\Pi_{E\parallel F}
|_W)^{-1}$. Let $0<2\epsilon<\lambda_j-\lambda_{j+1}$ and $C_\omega>0$ satisfy
for every $x\in E_j(\sigma^n\omega)$, and $u\in E$ such that $\LL^{(n)}_\omega u=x$,
$\|u\|\le C_\omega e^{-(\lambda_j-\epsilon)n}\|x\|$; and for every $f\in F$,
$\|\LL_\omega^{(n)}f\|\le C_\omega e^{(\lambda_{j+1}+\epsilon)n}\|f\|$.
Now $Qu-u=Qu-\Pi_{E|F}Qu
\in F$, so that 
$\|\LL_\omega^{(n)}(Qu-u)\|\le C_\omega e^{(\lambda_{j+1}+\epsilon)n}\|Qu-u\|$. Hence,
$\|\LL_\omega^{(n)}(Qu)-x\|\le C_\omega^2 e^{-(\lambda_j-\lambda_{j+1}-2\epsilon)n}(\|Q\|+1)\|x\|$.
Since $\LL_\omega^{(n)}(Qu)\in\LL_\omega^{(n)}W\subset \LL_\omega^{(n)}V$, the proof is complete.
\end{proof}

\begin{lem}\label{lem:insidefp}
Let the measure-preserving transformation and cocycles satisfy
$r_{\mathcal T}(R)<R$  as in the statement of Theorem
\ref{thm:LyapSpect}. Let $V$ be the subspace of $H^2(A_R)$
spanned by the Laurent polynomials $z^{-(j+1)}$ for $j=1,\ldots,N$.
Then $\LL_\omega^{(n)}V$ is spanned by $1/\big(z-T_\omega^{(n)}(0)\big)^{j+1}$
 for $j=1,\ldots,N$. 
 
 In particular, the sequence $\LL_\omega^{(n)}V$ approaches
 the equivariant sequence of subspaces
 $$
 P^-_{N}(\sigma^n\omega):=\lin\left\{
 \frac{1}{(z-x_{\sigma^n\omega})^{j+1}}\colon 1\le j\le N\right\}.
 $$ 
 \end{lem}

\begin{proof}
It suffices to show that if $f(z)=1/(z-x)^{j+1}$ with $x\in D$,
then for any finite expanding Blaschke product,
$\LL_\omega f$ is a linear combination of $1/(z-T(x))^{k+1}$ for $k$ 
in the range 1 to $j$, but that was established in Corollary \ref{cor:uppertri}.
\end{proof}

\begin{cor}\label{cor:outsidefp}
Let the measure-preserving transformation and cocycles
be as above. Let $W$ be the subspace of $H^2(A_R)$
spanned by the Laurent polynomials $z^{j-1}$ for $j=1,\ldots,N$.
Then $\LL_\omega^{(n)}W$ is spanned by 
$z^{j-1}/\big(1-\overline{T_\omega^{(n)}(0)}z)\big)^{j+1}$
for $j=1,\ldots,N$.

In particular,  the sequence $\LL_\omega^{(n)}W$ approaches
the equivariant sequence of subspaces,
$$
P^+_N(\sigma^n\omega)=
\lin\left\{\frac {
z^{j-1}}{(1-\bar x_{\sigma^n\omega} z)^{j+1}}\colon 1\le j\le N\right\}.
$$
\end{cor}

\begin{proof}
Notice that $\LL_I$ maps $z^{-(j+1)}$ to $z^{j-1}$ (and vice versa), and is a continuous
operator on $H^2(A_R)$. So
$\LL_I(V)=W$, where $V$ is as in the statement of Lemma 
\ref{lem:insidefp}. 
Since $\LL_\omega$ and $\LL_I$ commute, we see $\LL_\omega^{(n)}W
=\LL_\omega^{(n)}\LL_I(V)=\LL_I(\LL_\omega^{(n)}(V))$. A computation shows
that if $f(z)=1/(z-x)^{j+1}$, then $\LL_If(z)=z^{j-1}/(1-\bar xz)^{j+1}$. 
\end{proof}

\begin{cor}\label{cor:fast}
Let the dynamical system, Blaschke product cocycle and family of 
Perron-Frobenius operators satisfy $r_{\mathcal T}(R)<R$ as above. 
Let $E(\omega)$ be an equivariant family of finite-dimensional fast spaces for the cocycle. 
Then there exists an $N$ such that for $\PP$-a.e. $\omega$,
$$
E(\omega)\subset P^-_N(\omega)\oplus W_0(\omega)\oplus P^+_N(\omega),
$$
where $W_0(\omega)$ is as defined in Corollary \ref{cor:topspace}.
\end{cor}

\begin{proof}
Let $F(\omega)$ be the corresponding slow subspace to $E(\omega)$. 
Since the (finite term) Laurent polynomials, $L$, form a dense subspace of $H^2(A_R)$, 
and $\Pi_{E(\omega)\parallel F(\omega)}$ is bounded, we see that $\Pi_{E(\omega)\parallel F(\omega)}(L)
$ is a dense subspace of $E(\omega)$ and hence is equal to $E(\omega)$. 
Pick out a finite-dimensional subspace $L_1$ of $L$ such that $\Pi_{E(\omega)\parallel 
F(\omega)}(L_1)=E(\omega)$. 
Hence there exists an
$N$ such that $V=\lin\{z^{-1+j}\colon |j|\le N\}$ satisfies
the hypothesis of Lemma \ref{lem:fastapprox2}.
By Lemma \ref{lem:insidefp} and Corollaries \ref{cor:topspace} and
\ref{cor:outsidefp}, we see 
\begin{equation}\label{eq:degN}
\sup_{x\in E(\sigma^n\omega)\cap S(X)}
d(x,P_N^-(\sigma^n\omega)\oplus W_0(\sigma^n\omega)\oplus 
P_N^+(\sigma^n\omega))\to 0.
\end{equation}

For a fixed $N$, let $A_N$ be the set of $\omega$ for which 
\eqref{eq:degN} is satisfied and 
notice that $A_N$ is a $\sigma$-invariant measurable subset of $\Omega$. 
Hence there exists an $N>0$
such that for $\PP$-a.e. $\omega$,
$$
\sup_{x\in E(\sigma^n\omega)\cap S(X)}
d(x,P_N^-(\sigma^n\omega)\oplus W_0(\sigma^n\omega)\oplus 
P_N^+(\sigma^n\omega))\to 0.
$$

It follows from the Poincar\'e recurrence theorem that if 
$\kappa\colon\Omega\to[0,\infty)$ is 
a measurable function such that $\kappa(\sigma^n\omega)\to 0$ 
for $\PP$-a.e.\ $\omega$, 
then $\kappa(\omega)=0$ for almost every $\omega$. 
We apply this to
$$
\kappa(\omega)=\sup_{x\in E(\omega)\cap S(X)}
d(x,P_N^-(\omega)\oplus W_0(\omega)\oplus 
P_N^+(\omega)),
$$
to deduce that
$E(\omega)\subset P_N^-(\omega)\oplus W_0(\omega)\oplus P_N^+(\omega)$
for $\PP$-a.e. $\omega$, as required.
\end{proof}

\begin{proof}[Proof of Theorem \ref{thm:LyapSpect}]
The compactness of the cocycle follows from Corollary \ref{cor:compact} and 
statement \eqref{it:randomfp} follows from Corollary \ref{cor:Blaschkefp}. 

In the light of  Corollary \ref{cor:fast}, 
it suffices to evaluate the Lyapunov exponents when the
system is restricted to the finite-dimensional equivariant subspaces 
$P_N(\omega)=P_N^-(\omega)\oplus W_0(\omega)\oplus P_N^+(\omega)$. 
Notice that since each of $E_0(\omega)$ and $P_N^\pm(\omega)$ is
equivariant, the Lyapunov exponents of the cocycle restricted to $P_N$
are simply the combination of the Lyapunov exponents of $E_0(\omega)$,
$P_N^+(\omega)$ and $P_N^-(\omega)$. Since $\LL_I(P_N^\pm(\omega))=
P_N^\mp(\omega)$, $\LL_I$ is a bounded involution, and $\LL_{T_\omega^{(n)}}
\circ \LL_I=\LL_I\circ\LL_{T_\omega^{(n)}}$, we deduce the Lyapunov exponents
of the restriction of the cocycle to $P^+_N(\omega)$ are the same as those 
of the restriction to $P^-_N(\omega)$. 
As noted in Corollary~\ref{cor:topspace}, the exponent of the cocycle 
restricted to the equivariant 
space $E_0(\omega)$ is 0 (this will turn out to be the leading exponent). 
It suffices to compute the Lyapunov exponents of the restriction of the
cocycle to $P_N^-(\omega)$. Each of these Lyapunov exponents will then 
have multiplicity two for the full cocycle, being repeated as a Lyapunov exponent
in the restriction to $P_N^+(\omega)$.

It follows from Corollary \ref{cor:uppertri} that the matrix representing 
the restriction of the cocycle
to $P_N^-(\omega)$ is upper triangular with respect to the natural 
family of bases, $(z-x_\omega)^{-(j+1)}$ for $j=1,\ldots,N$. 

If $T'(x_\omega)=0$,
then all the diagonal terms of the matrix are 0 by Theorem \ref{thm:meroPF}.
Hence if $T_\omega'(x_\omega)=0$ for
a set of $\omega$'s of positive measure, we see that the Lyapunov spectrum 
is 0 with multiplicity 1 and $-\infty$ with infinite multiplicity.

Otherwise, we compute the leading term of
$\LL_\omega f(z)$ near $x_{\sigma\omega}$, where $f(z)=1/(z-x_\omega)^{j+1}$. 
Let $\alpha=T_\omega'(x_\omega)$. 
Then we have
\begin{align*}
\LL_\omega f(x_{\sigma\omega}+h)&=
\frac {f(x_\omega+h/\alpha)}{T_\omega'(x_\omega)}+O(h^{-j})\\
&=(\alpha/h)^{j+1}/\alpha+O(h^{-j})=\alpha^j/h^{j+1}+O(h^{-j}).
\end{align*}
That is, the diagonal entry of the matrix is $\big(T_\omega'(x_\omega)\big)^j$.
We also verify that the off-diagonal elements of the matrix are bounded:
If $i<j$, then the $(i,j)$ entry of the matrix is given by 
$\frac{1}{2\pi i}\int\LL_\omega[f](z)(z-x_{\sigma\omega})^i\,dz$.
where $f(z)=(z-x_\omega)^{-(j+1)}$ and the integral is over the unit circle.
Since by Corollary \ref{cor:compact},
the operators $\LL_\omega$ are a uniformly bounded family on $H^2(A_R)$, we see
that the entries of the $N\times N$ matrix, representing the restriction to $P_N^-(\omega)$,
are uniformly bounded. (In fact, we give more refined estimates
in Section \ref{sec:Lyapstab}.)

The Lyapunov exponents of the cocycle restricted to $P^+_N(\omega)$
are therefore given by the values
\begin{align*}
&\lim_{n\to\infty}\frac 1n\log\prod_{k=0}^{n-1}
\big|T_{\sigma^k\omega}'(x_{\sigma^k\omega})\big|^j\\
&=j\lim_{n\to\infty}\frac 1n\sum_{k=0}^{n-1}\log 
\big| T_{\sigma^k\omega}'(x_{\sigma^k\omega})\big|\\
&=j\int \log \big| T_\omega'(x_\omega)\big|\,d\PP(\omega)=:j\Lambda,
\end{align*}
where $j$ ranges from 1 to $N$, and we used the Birkhoff ergodic theorem in the last line.

Finally, to show that $\Lambda\le \log\frac rR$, notice that the restriction 
of $d_R$ to $\overline{D_r}$
agrees with Euclidean distance up to a bounded factor. The above shows that $\Lambda=
\lim_{n\to\infty}\log|{T_\omega^{(n)}}'(x_\omega)|$.
Lemma \ref{lem:contraction} shows that 
$d_R(T_\omega^{(n)}(x_\omega+h),T_\omega^{(n)}(x_\omega))\le a|h|(r/R)^n$,
where $a=d_R(x_\omega+h,x_\omega)/|h|$ is a uniformly bounded quantity. 
Hence $|T_\omega^{(n)}(x_\omega+h)-T_\omega^{(n)}(x_\omega)|/h \le c(r/R)^n$. 
The fact that $\Lambda\le\log\frac rR$ follows.
\end{proof}

\section{Spectrum Collapse}\label{sec:phasetrans}
In this section, we focus on an example. 
Let $T_0(z)=z^2$ and $T_1(z)=\left(\frac{z+1/4}{1+z/4}\right)^2$, so that 
both $T_0$ and $T_1$ are expanding degree 2 maps of the unit circle, mapping the unit
disc to itself in a two-to-one way. We take the base
dynamical system to be the full shift $\sigma$ on $\Omega=\{0,1\}^\Z$
with invariant measure $\PP_p$, the Bernoulli measure where each
coordinate takes the value 0 with probability $p$ and 1 with probability $1-p$. 

We let $\mathcal L_0$ and $\mathcal L_1$ be the Perron-Frobenius operators 
corresponding to $T_0$ and $T_1$ acting on the unit circle with respect to the
signed measure, $dz$, and consider the cocycle 
$\LL_\omega:=\LL_{\omega_0}$ and study the properties
of $\LL^{(n)}_\omega:=\LL_{\omega_{n-1}}\circ\cdots\circ\LL_{\omega_0}$.

\begin{lem}\label{lem:randomfp}
Let $T_0$ and $T_1$ be defined as above. Then 
\begin{enumerate}[label=(\alph*)]
\item $T_0$ fixes 0 and $T_1$ fixes $a=\frac 12(7-3\sqrt 5)
\approx 0.146$;
\item $T_0$ and $T_1$ both map the subset $[0,a]$ of the unit disk
in a monotonically increasing way into itself (with disjoint ranges);
\item \label{it:partc}both maps act as contractions on [0,a]:
$\frac{15}{32}\le T_1'\le \frac 23$ on $[0,a]$ and $0\le T_0'\le 2a$ on $[0,a]$.
\end{enumerate}

For $\omega\in\Omega$, let $x_\omega$ denote the random fixed point
as described in Theorem \ref{thm:LyapSpect}.

\begin{enumerate}[resume, label=(\alph*)]
\item
If $\omega=\ldots 10^n\cdot 0\ldots$, then $2b^{2^n}\le 
T_{\omega_0}'(x_\omega)\le
2a^{2^n}$, where $b=T_1(0)$.\label{it:partd}
\end{enumerate}
\end{lem}

\begin{proof}
We just prove statement \ref{it:partd}.
Let $\omega_{-(n+1)}=1$ and $\omega_{-n}=\ldots=\omega_{-1}=\omega_0=0$.
Then since $x_{\sigma^{-n}\omega}=T_1(x_{\sigma^{-(n+1)}\omega})$, we have
$b\le x_{\sigma^{-n}\omega}\le a$. 
Since $x_\omega=T_0^n(x_{\sigma^{-n}\omega})$, we have
$b^{2^n}\le x_\omega\le a^{2^n}$ and
$2b^{2^n}\le T_\omega'(x_\omega)\le 2a^{2^n}$.
\end{proof}

\begin{proof}[Proof of Corollary \ref{cor:unpert}]
%
For \ref{it:finite}, using Theorem \ref{thm:LyapSpect}, it suffices to prove
$\Lambda>-\infty$, where $\Lambda=\int\log|T_\omega'(x_\omega)|\,d\PP(\omega)$.
$\Omega$ may be countably partitioned (apart from the fixed point of all 0's) into 
$[\cdot1]:=\{\omega\in\Omega\colon \omega_0=1\}$ and the sets
$[10^n\cdot 0]:=\{\omega\in\Omega\colon x_{-(n+1)}=1,x_{-n}=\ldots=x_0=0\}$
for $0\le n<\infty$.
On $[\cdot 1]$, by Lemma \ref{lem:randomfp}\ref{it:partc}, 
$\log T_\omega'(x_\omega)\ge\log\frac{15}{32}$.
On $[10^n\cdot 0]$, $\log T_\omega'(x_\omega)\ge2^n\log b$ by 
Lemma \ref{lem:randomfp}\ref{it:partd}.
Since $\PP([10^n\cdot 0])=(1-p) p^{n+1}$, we see
\begin{align*}
\int\log T_\omega'(x_\omega)\,d\PP(\omega)&=
\int_{[\cdot 1]}\log T_\omega'(x_\omega)\,d\PP
+\sum_{n=0}^\infty \int_{[10^n\cdot 0]}\log T_\omega'(x_\omega)\,d\PP\\
&\ge (1-p)\log\tfrac{15}{32}+p(1-p)\log b\sum_{n=0}^\infty (2p)^n>-\infty.
\end{align*}

For \ref{it:infinite}, arguing as above, we see that on $[10^n\cdot 0]$,
$\log |T_\omega'(x_\omega)|\le 2^n\log a+\log 2$ (where $\log a\approx -1.925$).
Hence 
\begin{align*}
\Lambda&\le \PP_p([1])(\log \tfrac23)+ \sum_{n=0}^\infty
(2^n\log a+\log 2)\PP_p([10^n\cdot 0])\\
&=(1-p)\log\tfrac 23+p\log 2+p(1-p)\log a\sum_{n=0}^\infty (2p)^n=-\infty.
\end{align*}
\end{proof}

\subsection{Gaussian perturbations}\label{subsec:normalpert}
We now consider a perturbed version of the cocycle, where $\mathcal L_i$ is replaced by
$\LL^\epsilon_i:=\mathcal N_\epsilon\circ\LL_i$, where 
$\mathcal N_\epsilon$ has the effect
of convolving a density with a Gaussian with mean 0 and variance $\epsilon^2$. 
On $\R/\Z$, we have
$$
(\mathcal N^{\R/\Z}_\epsilon f)(x)=\frac{1}{\sqrt{2\pi}}\int_{-\infty}^\infty 
f(x-\epsilon t)e^{-t^2/2}\,dt=\mathbb Ef(x+\epsilon N),
$$
where $N$ is a standard normal random variable. 
The corresponding conjugate operator on $C(C_1)$ is $\mathcal N_\epsilon:=
\mathcal N^{C_1}_\epsilon=Q^{-1}\mathcal N_\epsilon^{\R/\Z}Q$, 
where $Q$ is as in Lemma \ref{lem:corr}. A calculation using that lemma shows
$$
(\mathcal N_\epsilon f)(z)=\frac{1}{\sqrt{2\pi}}\int_{-\infty}^\infty 
f(ze^{-2\pi i\epsilon t})e^{-2\pi i\epsilon t-t^2/2}\,dt.
$$

\begin{proof}[Proof of Corollary \ref{cor:pert}]
From the definition of $\LL_0$, we check
$$
\LL_0(f)(z)=\frac 12\left(\frac {f(\sqrt z)}{\sqrt z}+\frac{f(-\sqrt z)}{-\sqrt z}\right),
$$
where $\pm\sqrt z$ are the two square roots of $z$. We define 
$\hat{e}_n(z)=z^{n-1}$ and verify that
\begin{equation}\label{eq:L0}
\LL_0(\hat{e}_n)=\begin{cases}
\hat{e}_{n/2}&\text{if $n$ is even;}\\
0&\text{otherwise.}
\end{cases}
\end{equation}

We compute 
\begin{align*}
\NN_\epsilon(\hat{e}_n)(z)&=
\frac{1}{\sqrt{2\pi}}\int_{-\infty}^\infty \hat{e}_n(ze^{-2\pi i\epsilon t})
e^{-2\pi i\epsilon t-t^2/2}\,dt\\
&=z^{n-1}\frac{1}{\sqrt{2\pi}}\int_{-\infty}^\infty
e^{-2\pi in\epsilon t-t^2/2}\,dt\\
&=e^{-2\pi^2n^2\epsilon^2}\hat{e}_n(z).
\end{align*}

Combining the two, we have
\begin{align*}
(\LL_0^\epsilon)^n\hat{e}_\ell=
\begin{cases}
\exp(-2\pi^2\epsilon^2m^2(4^{n-1}+\ldots+4+1))\hat{e}_m&\text{if $\ell=2^nm$;}\\
0&\text{otherwise.}
\end{cases}
\end{align*}

We let $H^2_0(A_R)$ be the subspace of $H^2(A_R)$ consisting of those 
functions whose Laurent expansions have a vanishing $z^{-1}$ term.
Let $f\in H^2_0(A_R)$ be of norm 1 and let $f=\sum_{n\in\Z}a_nz^n$
be its Laurent expansion. We recall $|a_n|\le R^{|n|}\le 1$ 
for all $n\in\Z$ and $a_{-1}=0$.

Now
$$
(\LL_0^\epsilon)^nf(z)=\sum_{m\in\Z\setminus\{0\}}
\exp(-2\pi^2\epsilon^2m^2(4^{n-1}+\ldots+4+1))a_{2^nm-1}z^{m-1}.
$$
so that for $z\in A_R$ and $n>0$, 
\begin{align*}
|(\LL_0^\epsilon)^nf(z)|
&\le \sum_{m\in\Z\setminus\{0\}}
\exp(-2\pi^2\epsilon^2m^2(4^{n-1}+\ldots+4+1))
R^{|2^nm-1|}R^{-|m-1|}\\
&\le \tfrac {2}{1-R}\exp(-2\pi^2\epsilon^24^{n-1}).
\end{align*}

%
%

Since if $g$ is a bounded analytic function on $A_R$, 
$\|g\|_{H^2(A_R)}\le 2\|g\|_\infty$, we see
$\big\|(\LL_0^\epsilon)^n\vert_{H^2_0(A_R)}\big\|
\le \tfrac {4}{1-R}\exp\big(-2\pi^2\epsilon^24^{n-1}\big)$.
By Lemma \ref{lem:exptsviavol}\ref{it:Dfactor},
$\mathcal D_2((\LL_0^\epsilon)^n\LL_1^\epsilon)\le A
\exp\big(-2\pi^2\epsilon^2 4^{n-1}\big)$, where $A=4\|\LL_1\|^2c_2/(1-R)$.

Now let $N(\omega)=\min\{n>0\colon\omega_n=1\}$. We consider the induced map
on $[1]$: $\tilde\sigma(\omega)=\sigma^{N(\omega)}(\omega)$. The induced cocycle 
is defined for $\omega\in[1]$ by $\tilde\LL^\epsilon_\omega
={\LL^\epsilon_\omega}^{(N(\omega))}$, so that
$\tilde\LL^\epsilon_\omega=(\LL^\epsilon_0)^{N(\omega)-1}\LL_1^\epsilon$.
By ${\tilde{\mathcal L}_\omega^\epsilon}{}^{(n)}$, we mean 
$\tilde\LL^\epsilon_{\tilde\sigma^{n-1}(\omega)}\circ\cdots\circ
\tilde\LL^\epsilon_\omega$ and by $\tilde \PP$, we mean the normalized
restriction of $\PP$ to $[1]$ 
(so the convention is that quantities marked with
tildes refer to the induced system).

We define the return times for $\omega\in[1]$ by $N_1(\omega)=N(\omega)$
and $N_{n+1}(\omega)=N_n(\omega)+N(\sigma^{N_n(\omega)}(\omega))$ for $n\ge 1$.
Now we have, using Lemma \ref{lem:exptsviavol}\ref{it:Dsubmult},
\begin{align*}
&\frac{1}{N_n(\omega)}\log \mathcal D_2({\LL_\omega^\epsilon}^{(N_n(\omega))})=
\frac 1{N_n(\omega)}\log \mathcal D_2({{\tilde\LL}_\omega^\epsilon}{}^{(n)})\\
&=\frac{n}{N_n(\omega)}\frac{1}n\log 
\mathcal D_2(\tilde\LL^\epsilon_{\tilde\sigma^{n-1}\omega}\circ
\ldots\circ\tilde\LL_\omega^\epsilon)\\
&\le\frac{n}{N_n(\omega)}\frac 1n\log\left(
\mathcal D_2(\tilde\LL_{\tilde\sigma^{n-1}\omega}^\epsilon)\cdot\ldots\cdot
\mathcal D_2(\tilde\LL_{\omega}^\epsilon)\right)\\
&=\frac{n}{N_n(\omega)}\frac 1n\sum_{i=0}^{n-1}
\log \mathcal D_2(\tilde\LL^\epsilon_{\tilde\sigma^i\omega})\\
&\le\left(\frac{n}{N_n(\omega)}\right)\frac 1n\sum_{i=0}^{n-1}
(-2\pi^2\epsilon^2 4^{N(\tilde\sigma^i\omega)-1}+\log A).
\end{align*}
Since $\int_{[1]} 4^{N(\omega)}\,d\tilde\PP_p(\omega)=
\sum_{n=1}^\infty 4^np^{n-1}(1-p)=\infty$,
we see the average $\frac 1n\sum_{i=0}^{n-1}
(-2\pi^2\epsilon^2 4^{N(\tilde\sigma^i\omega)-1}+\log A)$
in the last line converges to $-\infty$
almost surely by Birkhoff's theorem
applied to the ergodic transformation $\tilde\sigma$ of $([1],\tilde\PP_p)$.
As $n/N_n(\omega)\to 1/\PP_p([1])$ for $\tilde\PP_p$-almost every $\omega\in[1]$,
we see 
$\frac{1}{N_n(\omega)}\log \mathcal D_2({\LL_\omega^\epsilon}^{(N_n(\omega))})
\to-\infty$ for $\tilde\PP_p$-a.e. $\omega\in[1]$. Since this is a subsequence of the
convergent sequence $\frac 1n\log \mathcal D_2({\LL_\omega^\epsilon}^{(n)})$, 
we see that $\frac 1n\log \mathcal D_2({\LL^\epsilon_\omega}^{(n)})\to-\infty$ 
for $\PP_p$-a.e. $\omega$.

This establishes, by Lemma \ref{lem:exptsviavol}\ref{it:Dexpts},
that $\lambda_2=-\infty$.  We recall from \cite[Theorem~13]{GTQ-JMD})
that the exceptional Lyapunov exponents of a cocycle and its adjoint coincide. 
We note that $\phi(f)=\langle f,\hat e_{-1}\rangle$ satisfies $\phi(f)=\phi(\LL_\omega^\epsilon f)$,
so that ${(\LL_\omega^\epsilon)}^*\phi=\phi$. Hence $\lambda_1=\lambda_1^*=0$. 
We briefly explain how  to identify the family of equivariant functions, or top 
Oseledets spaces $V_1(\omega)$ for $\LL^\epsilon$. 
For $\omega\in
\Omega$ and $\mathbf t=(t_n)_{n\in\Z^-}\in\R^{\Z^-}$, we define
$$
\Phi(\omega,\mathbf t)=\lim_{k\to\infty}
R_{2\pi\epsilon t_{-1}}T_{\sigma^{-1}\omega}\circ\cdots\circ
R_{2\pi\epsilon t_{-k}}T_{\sigma^{-k}\omega}(0),
$$
where $R_\theta(z)=e^{i\theta}z$.
Existence of the limit follows as in Corollary \ref{cor:Blaschkefp}.
The equivariant function is then given by
$$
f_\omega(z)=\int_{\R^{\Z^-}}\left(
\frac{1}{z-\Phi(\omega,\mathbf t)}
-\frac{1}{z-I(\Phi(\omega,\mathbf t))}\right)\,d\Gamma(\mathbf t),
$$
where $\Gamma$ is the measure on $\R^{\Z^-}$ where each coordinate is an 
independent standard normal random variable. The proof that $\LL_\omega
f_\omega=f_{\sigma\omega}$ essentially
follows from Lemma \ref{lem:simplepoleit} and Corollary \ref{cor:topspace}.
Since the range of $\Phi$ is contained in 
$\overline{D_r}$, it is easy to see that $f_\omega$ lies in $H^2(A_R)$ for all
$\omega\in\Omega$.
\end{proof}

\subsection{Uniform perturbations}
We now consider another perturbed version of the cocycle, where 
$\mathcal L_i$ is replaced by
$\LL^{U,\epsilon}_i:=\mathcal U_\epsilon\circ\LL_i$, where 
$\mathcal U_\epsilon^{\R/\Z}$ has the effect
of convolving a density with a bump function of support
$[-\epsilon, \epsilon]$. 
On $\R/\Z$, we have
$$
(\mathcal U^{\R/\Z}_\epsilon f)(x)=\frac{1}{2\epsilon}\int_{-\epsilon}^\epsilon 
f(x- t)\,dt = \frac{1}{2}\int_{-1}^1 
f(x-\epsilon t)\,dt=\mathbb Ef(x+\epsilon U),
$$
where $U$ is a uniformly distributed random variable on $[-1,1]$. 
The corresponding conjugate operator on $C(C_1)$ is $\mathcal U_\epsilon:=
\mathcal U^{C_1}_\epsilon=Q^{-1}\mathcal U_\epsilon^{\R/\Z}Q$, 
where $Q$ is as in Lemma \ref{lem:corr}. A calculation shows
$$
(\mathcal U_\epsilon f)(z)=\frac{1}{2}\int_{-1}^1 
f(ze^{-2\pi i\epsilon t})e^{-2\pi i \epsilon  t}\,dt.
$$
As before, we let $\hat{e}_n(z)=z^{n-1}$ and compute, for $n \in \Z \setminus \{0\}$ 
\begin{align*}
\mathcal U_\epsilon(\hat{e}_n)(z)&=
\frac{1}{2}\int_{-1}^1 \hat{e}_n(ze^{-2\pi i\epsilon t})
e^{-2\pi i\epsilon t}\,dt\\
&=\frac{z^{n-1}}{2}\int_{-1}^1
e^{-2\pi in\epsilon t}\,dt\\
&=\frac{\sin(2\pi  n \epsilon)}{2 \pi  n \epsilon}  \hat{e}_n(z).
\end{align*}
Also, $\mathcal U_\epsilon(\hat{e}_0)(z)=\hat{e}_0(z)$.
Hence, we have, for $\ell \in \Z \setminus \{0\}$,
\begin{align*}
(\LL_0^{U,\epsilon})^n\hat{e}_\ell=
\begin{cases}
(2\pi m  \epsilon )^{-n} 2^{-n(n-1)/2}
{\prod_{j=1}^{n}\sin(2^jm\pi  \epsilon)}
\hat{e}_m&\text{if $\ell=2^nm$;}
\\
0&\text{otherwise.}
\end{cases}
\end{align*}
\begin{proof}[Proof of Corollary~\ref{cor:Unifpert}]
Let $\epsilon = b/2^k$, for some odd integer $b$ and $k \in \mathbb N$. Then, for every 
$\ell \in \Z \setminus \{0\}$ and $n\geq k$ we get that 
$(\LL_0^{U,\epsilon})^n\hat{e}_\ell=0$.
This immediately implies that
$(\LL_0^{U,\epsilon})^n \vert_{H^2_0(A_R)}=0$, and so 
 $\lambda_2 (\LL_\omega^{U, \epsilon}) = -\infty$.
 The fact that $\lambda_1(\LL_\omega^{U, \epsilon})=0$ follows 
 exactly as in Corollary~\ref{cor:pert}.
\end{proof}

\section{Lyapunov spectrum stability}\label{sec:Lyapstab}
It is natural to ask for an underlying explanation for the instability of the Lyapunov spectrum
exhibited in Corollaries \ref{cor:pert} and \ref{cor:Unifpert}. From the finite-dimensional
theory of hyperbolic dynamical systems, we know that a key issue in the stability is
control of the angle between the fast and slow subspaces. In general, the various
versions of the Multiplicative Ergodic Theorem show that for Oseledets spaces with
different exponents, the angle between the subspaces is bounded away from
0 at least by a quantity that is at worst sub-exponentially small in $n$, the number
of iterations. In the uniformly hyperbolic situation, this angle is uniformly bounded
away from 0.

One way to quantify the angle between complementary closed subspaces
that is particularly well suited to the infinite-dimensional case
is to compute $\|\Pi_{E\parallel F}\|$, where $\Pi_{E\parallel F}$ is the 
projection that fixes $E$ and annihilates $F$: for Hilbert spaces
the norm of the projection is the reciprocal of the sine of the angle between the spaces.
(See \cite{GTQ-JMD} for an alternative way of
measuring angles).

With this in mind, we study $\Pi_{E_k(\omega)\parallel F_k(\omega)}$ where
$E_k(\omega)$ is the span of the Oseledets vectors with exponents
$\lambda_1,\ldots,\lambda_k$ and $F_k(\omega)$ is the complementary space 
of vectors that expand at rate $\lambda_{k+1}$ or slower (the $(2k-1)$-dimensional
\emph{fast space} and $(2k-1)$-codimensional \emph{slow space} respectively). 

For the unperturbed cocycle appearing in Corollaries \ref{cor:pert} and
\ref{cor:Unifpert}, we claim that $\|\Pi_{E_k(\omega)\parallel
F_k(\omega)}\|$ is essentially unbounded in $\omega$. 
To see this, let $\sigma$ be the shift map on $\{0,1\}^\Z$ equipped with the
Bernoulli probability measure, $\PP_p$ and $(\mathcal L_\omega)$ as before.
Set $h_\omega(z)=(z-x_\omega)^{-2}$ and $g(z)=z^{-2}$. Note that
If $\omega_0=0$, then $\LL_\omega g=0$ (see \eqref{eq:L0}), so that
$g\in F_2(\omega)$. Also $h_\omega\in E_2(\omega)$ by Lemma \ref{lem:insidefp},
and if $\omega_{-n}=\ldots=\omega_{-1}=\omega_0=0$, then
$x_\omega\le a^{2^n}$, where $a<1$ is as in Lemma~\ref{lem:randomfp}.
In particular $\essinf_\omega\|h_\omega-g\|=0$. However, since $\Pi_{E_2(\omega)
\parallel F_2(\omega)}(h_\omega-g)=h_\omega$ and $\|h_\omega\|$ is bounded
away from 0, we see that $\|\Pi_{E_2(\omega)\parallel F_2(\omega)}\|$ is
essentially unbounded for the cocycle, so that the fast and slow spaces become
arbitrarily close.

The core of the issue is that the kernel of $\LL_0$ includes all even 
integer powers of 
$(z-0)$ (0 being the critical point of $T_0$) while combinations of
negative powers of $(z-x_\omega)$ appear in the fast spaces. To avoid the
situation above, one is led to consider situations in which the critical point(s) of $T_\omega$
are bounded away from the random fixed point $x_\omega$. We impose this by
assuming a lower bound on $|T_\omega'(x_\omega)|$. 

In this section, we show that if we impose the condition
that $|T_\omega'(x_\omega)|$ is bounded below,
then the projections 
$\Pi_{E_k(\omega) \parallel F_k(\omega)}$ are uniformly bounded, and hence
the fast and slow spaces are uniformly transverse.
We deduce the existence of a field 
of cones around $E_k(\omega)$, and use this to prove
Theorem \ref{thm:Lyapcont}, showing that the Lyapunov spectrum 
of the Perron-Frobenius cocycle is stable under small perturbations. 
Conversely, if $T'_\omega(x_\omega)$ is not bounded below, we 
show that small perturbations to the Perron-Frobenius cocycle,
even within the class of Perron-Frobenius operators of Blaschke products,
can lead to collapse of the Lyapunov spectrum.

In the first part of the section, we replace the Blaschke product cocycle, $(T_\omega)$
with a conjugate cocycle $(\tilde T_\omega)$, almost all of whose elements fix the origin. We first prove
the theorem in that context, and then show how the full theorem follows.


\subsection{Bounded projections}
Let $\sigma\colon (\Omega,\PP)\to(\Omega,\PP)$ be an invertible
ergodic measure-preserving transformation and let $\mathcal T=
(T_\omega)_{\omega\in\Omega}$ be a Blaschke product cocycle 
satisfying the following conditions:
\begin{enumerate}[label=(\alph*)]
\item
$T_\omega(0)=0$ for $\PP$-a.e. $\omega\in\Omega$;\label{it:zerofp}
\item
$\essinf_\omega |T_\omega'(0)|>0$;\label{it:noncrit}
\item
$\esssup_\omega r_{\mathcal T}(R)<R$;\label{it:exp}
\end{enumerate}

Notice that for Blaschke product cocycles satisfying these conditions,
Theorem \ref{thm:LyapSpect} applies and the quantity $\Lambda$
arising is finite (and negative), so that the Lyapunov exponents are
$\lambda_j=(j-1)\Lambda$, where $\lambda_1=0$ has multiplicity 1
and the remaining exponents have multiplicity 2. Notice that \ref{it:noncrit}
gives uniform control on $|T_\omega'(x_\omega)|$, while the condition 
$\Lambda>-\infty$ in
Theorem \ref{thm:LyapSpect} for a non-trivial Lyapunov spectrum gives
only average control on $|T_\omega'(x_\omega)|$.

The proof of Theorem \ref{thm:LyapSpect} (in the special case $x_\omega=0$)
shows that for $j>1$, one can identify a natural basis for $V_j(\omega)$
consisting of a Laurent polynomial $f_{\omega,j}$ with $z^{-2},\ldots,z^{-j}$
terms and its inversion $\LL_I f_{\omega,j}$ with $z^0,\ldots,z^{j-2}$ terms.


We set $r=r_{\mathcal T}(R)$. We shall also require that $|T_\omega''|$ is uniformly bounded 
above on $D_r$, but in fact, this condition is true automatically since
$$
T''(z)=
\frac{2!}{2\pi i}\int_{C_1} \frac{T(w)}{(w-z)^3}\,dw
$$
and $|T(w)|=1$ whenever $|w|=1$, so that 
$|T''(w)|\le 2/(1-r)^3$ on $D_r$. 

Let $\lambda^{(n)}_\omega$ denote $|(T_\omega^{(n)})'(0)|$
in all of what follows.

\begin{lem}[Random fixed point distortion estimate]\label{lem:bounddist}
Under the conditions above, there exists a $c_1$ such that for $\PP$-a.e.
$\omega\in\Omega$, for all $z\in C_R$, 
$$
|T_\omega^{(n)}(z)|\le c_1\lambda^{(n)}_\omega.
$$
\end{lem}

\begin{proof}
Let $\gamma_{\omega,n}(z)=|T_\omega^{(n)}z|/
\lambda_\omega^{(n)}$ and $r=r_{\mathcal T}(R)<R$. 
Then we have
\begin{align*}
\gamma_{\omega,n+1}(z)&=
\frac{|T_\omega^{(n+1)}z-0|}{\lambda^{(n+1)}_\omega}\\
&\le \frac
{\max_{y\in [0,T_\omega^{(n)}z]}|T_{\sigma^n\omega}'(y)|
\cdot |T_\omega^{(n)}z-0|}
{\lambda_\omega^{(n)}T_{\sigma^n\omega}'(0)}\\
&=\gamma_{\omega,n}(z)\max_{y\in [0,T_\omega^{(n)}z]}
|T_{\sigma^n\omega}'(y)|/|T_{\sigma^n\omega}'(0)|,
\end{align*}
where $[a,b]$ denotes the line segment joining $a$ and $b$.
Since we showed in Lemma \ref{lem:contraction} the existence of a $c$
such that for $\PP$-a.e.\ $\omega$, $|T_\omega^{(n)}(z)-0|
\le c(\frac rR)^{n-1}$, we see that for $y\in [0,T_\omega^{(n)}
(z)]$, using the uniform boundedness of 
$T_\omega''$, we obtain the inequality
$|T_{\sigma^{n}\omega}'(y)/T_{\sigma^{n}\omega}'
(0)-1|\le c'(\frac rR)^{n}$, so that $\gamma_{\omega,n}(z)$
is bounded essentially uniformly in $\omega$ and uniformly in $n$ and $z$ as $z$ runs over $C_R$.
\end{proof}

Let $H^2(A_R)^-$ be $\lin(\{z^{-j-1}\colon j>0\})$
and $H^2(A_R)^+$ be $\lin(\{z^{j-1}\colon j>0\})$. 
(We choose to offset the indices by 1 as we have seen that this is well suited
to the calculations in the rest of the paper).
We set $\hat e_j(z)=R^{|j-1|}z^{j-1}$, as a convenient rescaling of the standard
orthogonal basis of $H^2(A_R)$ described earlier (in particular their norms are
between 1 and 2).

\begin{lem}\label{lem:compact}
Let $\rho<R$ and let $f$ lie in the unit sphere of $H^2(A_\rho)^-$. Then
$f$ may be expanded as
$f(z)=\sum_{m>0} a_{-m}\hat e_{-m}$ where 
$|a_{-m}|\le (\frac\rho R)^m$.
In particular, the partial sums of the series for $f$ given above are convergent
in $H^2(A_R)$. 
\end{lem}


Let $U_k^-$ be the
subspace of $H^2(A_R)^-$ spanned by $\hat e_{-1},
\ldots,\hat e_{-(k-1)}$ and define $V_k^-(\omega)$ by
$$
V_k^-(\omega)=
\{f\in H^2(A_R)^-\colon 
\limsup \tfrac 1n\log\|\LL^{(n)}_\omega f\|\le k\Lambda\},
$$
so that $H^2(A_R)^-=U_k^-\oplus V_k^-(\omega)$.
In particular, $U_k^-$ is the span of those Oseledets vectors,
mentioned above,
with Lyapunov exponents $\lambda_2,\ldots,\lambda_k$ that
are polynomials in $z^{-1}$. 
Let $Q^-$ denote the orthogonal projection of $H^2(A_R)$
onto $H^2(A_R)^-$ and let $Q^+$ be the orthogonal projection
of $H^2(A_R)$ onto $H^2(A_R)^+=\lin(\{z^{n-1}\colon n\ge 1\})$. 
We write $H^2(A_R)^\pm$ for $H^2(A_R)^+\oplus H^2(A_R)^-$ and
note that $H^2(A_R)^\pm=\lin(z^{-1})^\perp$.
We now show that under the conditions 
above, $U_k^-$ and $V_k^-(\omega)$ are uniformly transverse.

\begin{lem}\label{lem:UH}
Let the Blaschke product cocycle satisfy conditions \ref{it:zerofp}, \ref{it:noncrit},
\ref{it:exp}. Then for each $k\in\N$, 
there exists a constant $M>0$ such that 
$\|\Pi_{U^-_k\parallel V^-_k(\omega)}\|\le M$ for all $\PP$-a.e.\ $\omega\in\Omega$.
\end{lem}

\begin{proof}
Let $c_1$ be as in 
Lemma \ref{lem:bounddist},
and let $\hat e_{-j}$ be as defined earlier.
We compute the matrix of $\LL_\omega$ with respect to the $(\hat e_{-j})$
basis. 

For $j>0$, write $\LL_\omega(\hat e_{-j})$ as 
$\sum_{1\le i\le j} a_{ij}\hat e_{-i}$
(such an expansion exists with no higher order terms by Corollary \ref{cor:uppertri}).
We now have for $0<i<j$, 
\begin{align*}
a_{ij}&=
\frac{1}{2\pi i}\int_{C_1}\LL_\omega(\hat e_{-j})(z)R^{-i-1}
z^i\,dz\\
&=\frac{R^{-i-1}}{2\pi i}\int_{C_1}\hat e_{-j}(z)(T_\omega(z))^i
\,dz\\
&=\frac{R^{j-i}}{2\pi i}\int_{C_R}
\frac{(T_\omega(z))^i}{z^{j+1}}\,dz
\end{align*}
For $i=j$, we have $a_{jj}=(\lambda_\omega)^j$ as shown in Theorem \ref{thm:LyapSpect}.

Similarly, using Lemma \ref{lem:bounddist},
we see that for a.e.\ $\omega$, $a^{(n)}_{\omega,ij}$, the coefficient of 
$\LL^{(n)}_\omega$ with respect to $(\hat e_{-j})$ and 
$(\hat e_{-i})$ satisfies
\begin{equation}\label{eq:coeffbound}
\begin{split}
|a^{(n)}_{\omega,ij}|&=\left|\frac{R^{j-i}}{2\pi i}\int_{C_R}
\frac{(T_\omega^{(n)}(z))^i}
{z^{j+1}}\,dz\right|\\
&\le (c_1\lambda^{(n)}_\omega/R)^i\text{\quad for all $1\le i\le j$.}
\end{split}
\end{equation}
For $i=j$, $a^{(n)}_{jj}=(\lambda^{(n)}_\omega)^j$ and for $i>j$, 
$a^{(n)}_{ij}=0$. Letting $c_2=(c_1/R)^{k-1}$,
we have for a.e.\ $\omega$,
\begin{equation}\label{eq:entrybound}
|a^{(n)}_{\omega,ij}|\le c_2(\lambda_\omega^{(n)})^i\text{\quad  for all $1\le i\le k-1$ and all $j\in\N$.}
\end{equation}

Define $\Pi_{-k}$ to be the orthogonal projection from $H^2(A_R)^-$ onto $U_k^-$.
Now consider the operators $\Lambda_{\omega,k}^{(n)}
=(\LL^{(n)}_\omega\vert_{U_k^-})^{-1}
\circ\Pi_{-k}\circ\LL^{(n)}_\omega$.
We write
$$
\Lambda_{\omega,k}^{(n)}
=(\LL_\omega\vert_{U_k^-})^{-1}\circ
(\LL^{(n-1)}_{\sigma\omega}\vert_{U_k^-})^{-1}
\circ\Pi_{-k}\circ\LL^{(n-1)}_{\sigma\omega}\circ\LL_\omega.
$$

By \eqref{eq:entrybound}, the matrix representing the restricted operator
$(\LL^{(n-1)}_{\sigma\omega}\vert_{U_k^-})$ with respect to the 
bases $\hat e_{-1},\ldots,\hat e_{-(k-1)}$; and $\hat e_{-1},\ldots,\hat e_{-(k-1)}$ 
may be factorized as $DA$ where $D$ is the diagonal matrix with entries
$(\lambda^{(n-1)}_{\sigma\omega})^j$ for $j=1,\ldots,k-1$ and
$A$ is an upper triangular matrix with ones on the diagonal and with entries 
with absolute value bounded above by $c_2$ above the diagonal
for a.e.\ $\omega$.
Of course we have $(DA)^{-1}=A^{-1}D^{-1}$, and
by the above, $A^{-1}$ has uniformly bounded entries, and the matrix 
$B=(b_{ij})_{1\le i\le k-1; j\in\N}$
of $\big(\LL^{(n-1)}_{\sigma\omega}\vert_{U_k^-}\big)^{-1}
\Pi_{-k}\LL_{\sigma\omega}^{(n-1)}$ 
is $B=A^{-1}D^{-1}PL^{(n-1)}$,
where $P$ is the $(k-1)\times\infty$ matrix with ones on the diagonal and the remaining
entries 0 and $L^{(n-1)}$ is the matrix of $\LL_{\sigma\omega}^{(n-1)}$.
Combining \eqref{eq:entrybound} with the expression for $D$,
we obtain a constant $c_3$ so that for a.e.\ $\omega$ and all $n$, 
$|b_{ij}|\le c_3$ for $1\le i\le k-1$ and $j\in\N$. (Note that $c_3$ does depend on $k$). 

Let $r<\rho<R$.
By Corollary \ref{cor:compact}, there exists a constant $c_4$ such that
$\|\LL_\omega\|_{H^2(A_R)\to H^2(A_\rho)}\le c_4$ for almost all $\omega\in\Omega$.
Now if $f\in H^2(A_R)^-$ satisfies
$\|f\|_{H^2(A_R)}=1$, we have $\|\LL_\omega f\|_{H^2(A_\rho)}
\le c_4$, so that by Lemma \ref{lem:compact}, for $j>0$,
the coefficient of $\hat e_{-j}$ in the expansion of $\LL_\omega f$
is at most 
$c_4(\frac \rho R)^j$.

Combining this with the estimate for $b_{ij}$ above, 
we see that the $\hat e_{-i}$
coefficient of $(\LL^{(n-1)}_{\sigma\omega}|_{U_k^-})^{-1}
\Pi_{\sigma^n\omega,k}\LL^{(n)}_\omega f$ is bounded above by 
$c_3c_4\sum_{j=i}^\infty (\frac \rho R)^j=
c_3c_4(\frac\rho R)^ i/(1-\frac\rho R)$.
Since $(\LL_\omega\vert_{U_k^-})^{-1}$ is essentially uniformly bounded in $\omega$
(since it is upper triangular with uniformly bounded entries, and has diagonal elements 
uniformly bounded away from 0), we deduce the operators $\Lambda^{(n)}_{\omega,k}$
are uniformly bounded in $n$ and $\omega$ (but not in $k$), say 
$\|\Lambda^{(n)}_{\omega,k}\|\le M$ for all $\omega\in\Omega$ and $n\in\N$. 
It is immediate that $\Lambda^{(n)}_{\omega,k}f=f$ for all $f\in U_k^-$
and it is not hard to verify from the definition of $\Lambda^{(n)}_{\omega,k}$ that 
$\Lambda^{(n)}_{\omega,k}f\to 0$ for $f\in V_k^-(\omega)$. 
Hence $\Lambda^{(n)}_{\omega,k}$
converges strongly to $\Pi_{U_k^-\parallel V_k^-(\omega)}$, so that 
$\|\Pi_{U_k^-\parallel V_k^-(\omega)}\|\le M$.
\end{proof}

We define further spaces: $U_k^+=\LL_I U_k^-$,
$V_k^+(\omega)=\LL_I V_k^-(\omega)$ and $W_0=
\lin(\{\hat e_{0}\})$. 
Since $\LL_I$ commutes with $\LL_\omega$, the $\LL_\omega$
equivariance of $U_k^+$ implies that of $U_k^-$
and similarly equivariance of $V_k^+(\omega)$ follows from that of 
$V_k^-(\omega)$. The equivariance of $W_0$ was noted in
Corollary \ref{cor:topspace} (note that $x_\omega=0$ for the
cocycles we are considering). Since $H^2(A_R)^+=\LL_I(H^2(A_R)^-)$,
we see that $H^2(A_R)^+=U_k^+
\oplus V_k^+(\omega)$.

\begin{cor}
There exists $M'>0$ such that 
$\|\Pi_{U_k^+\parallel V_k^+(\omega)}\|\le M'$ 
for all $\omega\in\Omega$.
\end{cor}

\begin{proof}
We can verify that
$\Pi_{U_k^+\parallel V_k^+(\omega)}
=\LL_I\circ \Pi_{U_k^-\parallel V_k^-(\omega)}\circ\LL_I$.
Since $\LL_I$ is bounded, the result follows. 
\end{proof}

Next, let $E_k(\omega)=U_k^-\oplus W_0
\oplus U_k^+$ and $F_k(\omega)=V_k^-(\omega)
\oplus V_k^+(\omega)$. These are the $(2k-1)$-dimensional
and $(2k-1)$-codimensional fast and slow spaces for the cocycle
respectively. 

We observe that $\Pi_0:=\Pi_{W_0\parallel H^2(A_R)^\pm}$ is just
orthogonal projection onto $W_0$, so that $\|\Pi_{W_0\parallel
H^2(A_R)^\pm}\|=1$. 


\begin{cor}
There exists $M>0$ such that for all $\omega\in\Omega$,
$\|\Pi_{E_k(\omega)\parallel F_k(\omega)}\|\le M$. 
\end{cor}

\begin{proof}
We have
$$
\Pi_{E_k(\omega)\parallel F_k(\omega)}=
\Pi_0
+
\Pi_{U_k^-\parallel V_k(\omega)^-}\circ
Q^-
+
\Pi_{U_k^+\parallel V_k(\omega)^+}\circ
Q^+.
$$
Since all of the operators appearing on the right are uniformly bounded in $\omega$, 
so is $\Pi_{E_k(\omega)\parallel F_k(\omega)}$.
\end{proof}

\subsection{Invariant cone field}

In this sub-section, we shall show that there is a cone around $E_k(\omega)$
attracting a neighbourhood in $H^2(A_R)$. Consider the cone
$$
\mathcal C_{\omega,\eta}=
\{f\in H^2(A_R)\colon \|\Pi_{F_k(\omega)\parallel E_k(\omega)}f\|\le
\eta\|\Pi_{E_k(\omega)\parallel F_k(\omega)}f\|\}.
$$

\begin{lem}\label{lem:cone}
Let $\sigma$ be an ergodic invertible measure-preserving transformation
of $(\Omega,\PP)$ and let the Blaschke product cocycle satisfy 
conditions \ref{it:zerofp}, \ref{it:noncrit},
and \ref{it:exp}. Then for each $k\in\N$, there exists an $N$
such that for each $n\ge N$ and a.e.\ $\omega\in\Omega$, 
$\LL^{(n)}_\omega(\mathcal C_{\omega,\eta})
\subset \mathcal C_{\sigma^n\omega,\frac\eta2}$ for all $\eta>0$.
\end{lem}

\begin{proof}
Let $\Pi_{-k}$ be as in the proof of Lemma \ref{lem:UH},
the orthogonal projection onto $\lin(\hat e_{-1},\ldots,\hat e_{-(k-1)})$,
 and let $\Pi_{k}$ be the orthogonal projection onto
$\lin(\hat e_1,\ldots,\hat e_{k-1})$. Let $S_k=\Pi_{-k}+\Pi_k$ and
let $\Pi_0$ be the orthogonal projection onto $\lin(\hat e_0)$.

We make the following claims:
\begin{enumerate}[label=(\roman*)]
\item There exists a constant $K>0$ such that for
a.e.\ $\omega\in\Omega$, $n\in\N$ and
$f\in U_k^-$, $\|\LL_\om^{(n)}f\|\ge K(\lambda_\omega^{(n)})^{k-1}\|f\|$;
\label{it:efast-}
\item There exists $c_5>0$ such that for a.e.\ $\omega\in\Omega$, $n\in\N$ and
$f\in E_k(\omega)$, 
$\|\LL_\om^{(n)}f\|\ge c_5(\lambda_\omega^{(n)})^{k-1}\|f\|$;
\label{it:efast}
\item
There exist $c>0$ and $n_0\in\N$ such that for a.e.\ $\omega\in\Omega$, 
$n\ge n_0$ and $f\in H^2(A_R)^-$, 
$\|(I-\Pi_{-k})\LL_\omega^{(n)}f\|\le c(\lambda_\omega^{(n)})^{k}
\|f\|$.
\label{it:fslow-}
\item
There exist $c_6>0$ and $n_0\in\N$ such that for a.e.\ $\omega\in\Omega$, 
$n\ge n_0$ and $f\in H^2(A_R)^\pm$, 
$\|(I-S_k)\LL_\omega^{(n)}f\|\le c_6(\lambda_\omega^{(n)})^{k}
\|f\|$.
\label{it:fslow}
\end{enumerate}

To establish \ref{it:efast-}, first notice that
there is a constant $c>1$ such that 
for all $\omega\in\Omega$ and all vectors $(a_1,\ldots,a_k)$,  
$$
\|(a_1,\ldots,a_{k-1})\|_2/c
\le \left\|\sum_{i=1}^{k-1}a_i \hat e_{-i}\right\|
\le c\|(a_1,\ldots,a_{k-1})\|_2,
$$
where $\|\cdot\|_2$ denotes the Euclidean norm on $\R^k$.
Hence it suffices to demonstrate that there exists $c>0$ such that
$\|L^{(n)}_\omega v\|\ge c(\lambda_\omega^{(n)})^{k-1}\|v\|$ for
all $v\in \R^k$, where $L^{(n)}_\omega$ is the matrix of $\LL_\omega^{(n)}$
with respect to $(\hat e_{-j})_{j=1}^{k-1}$.
This is equivalent to showing the existence of a $c>0$
such that $\|(L^{(n)}_\omega)^{-1}\|\le c(\lambda^{(n)}_\omega)^{-(k-1)}$
for all $n\in\N$ and $\omega\in\Omega$.
We previously showed that $L^{(n)}_\omega$ may be expressed as $DA$
where $D$ is the diagonal matrix with entries $(\lambda^{(n)}_\omega)^i$,
with $i$ going from 1 to $k-1$; and $A$ is an upper triangular matrix
with bounded entries and 1's on the diagonal. It follows that
$\|(L^{(n)}_\omega)^{-1}\|\le c(\lambda_\omega^{(n)})^{-(k-1)}$ as required,
so that we have demonstrated \ref{it:efast-}.

To show \ref{it:efast}, let $f\in E_k(\omega)$, and write $f=f^++f^0+f^-$ where
the components lie in $U_k^+$, $W_0$ and
$U_k^-$ respectively. Clearly at least one of the components has 
norm at least $\|f\|/3$.  Since $\LL_\omega^{(n)}f^\star=\Pi_\omega^\star\LL_\omega^{(n)}f$,
where $\star$ is each of $+$, $0$ or $-$,
we have 
$$
\|\LL_\omega^{(n)}f\|\ge \max
\left(
\|\LL_\omega^{(n)}f^+\|,\|\LL_\omega^{(n)}f^0\|,
\|\LL_\omega^{(n)}f^-\|
\right).
$$
Hence it suffices to show that for each of $+$, 0 and $-$ there exists a $c^\star>0$
such that for all $\omega\in\Omega$ and $n\in\N$,
\begin{equation}\label{eq:toshow}
\|\LL_\omega^{(n)}f\|\ge c^\star(\lambda^{(n)}_\omega)^{k-1}\|f\|
\end{equation}
for $f$ lying in $U_k^\star$.

The first of these was demonstrated in \ref{it:efast-}.
If $f\in W_0$, then $\LL_\omega^{(n)}f=f$,
so that $\|\LL_\omega^{(n)}f\|=\|f\|$. Hence  \eqref{eq:toshow}
is satisfied when $\star$ is 0. 
If $f\in U_k(\omega)^+$, then $\LL^{(n)}_\omega f=\LL_I\LL^{(n)}_\omega\LL_If$.
Since $\LL_I$ is a bounded involution, it follows that there exists $c>0$ such that
$\|\LL_If\|\ge c\|f\|$ for all $f\in H^2(A_R)$. Combining this with \ref{it:efast-}
establishes the required result for $f$ lying in $U_k(\omega)^+$.

We now demonstrate \ref{it:fslow-}.
Let $\rho$ be as chosen in the proof of Lemma \ref{lem:UH}
and let $f\in H^2(A_R)^-$ be of norm 1. 
We showed in the proof of Lemma \ref{lem:UH}
that the $\hat e_{-j}$ 
coefficient of $\LL_\omega f$ is at most $c_4(\frac\rho R)^j$. 
Applying $\LL_{\sigma\omega}^{(n-1)}$ and recalling the estimate
\eqref{eq:coeffbound}, the coefficient of 
$\hat e_{-i}$ is at most
$\sum_{j\ge i}c_4(\frac\rho R)^j
(c_1\lambda_{\sigma\omega}^{(n-1)}/R)^i$.
This gives the estimate

\begin{align*}
\left\|
(I-\Pi_{-k})\LL_\omega^{(n)}f\right\|
&\le
\sum_{i\ge k}\sum_{j\ge i}
c_4(\tfrac\rho R)^j(c_1\lambda_{\sigma\omega}^{(n-1)}/R)^i
\|\hat e_{-i}\|\\
&\le\frac{2c_4}{1-\tfrac \rho R} \sum_{i\ge k}
(\rho c_1\lambda_{\sigma\omega}^{(n-1)}/R^2)^i\\
&=\frac{2c_4}{1-\tfrac \rho R}
(\rho c_1/R^2)^{k}(\lambda_{\sigma\omega}^{(n-1)})^{k}
\sum_{i\ge 0}(\rho c_1\lambda_{\sigma\omega}^{(n-1)}/R^2)^i.
\end{align*}
Since for a.e.\ $\omega$, $\lambda^{(n)}_\omega\le (r/R)^n$ for all $n$, 
there exists an $n_0$ such that for all $n\ge n_0$ and a.e.\  
$\omega\in\Omega$, $\rho c_1\lambda_\omega^{(n-1)}/R^2<\frac 12$. 
Since $|T_\omega'(x_\omega)|$ is essentially uniformly bounded below,
we have $\lambda^{(n-1)}_{\sigma\omega}\le \lambda^{(n)}_\omega/
\essinf_\omega |T_\omega'(x_\omega)|$ a.e.
Now for $n\ge n_0$, we have for all $f$ lying in the unit sphere
of $H^2(A_R)^-$,
$$
\left\|(I-\Pi_{-k})\LL_\omega^{(n)}f\right\|
\le c(\lambda_\omega^{(n)})^{k},
$$
where $c=4c_4(\rho c_1/R^2)^{k}/\big((1-\tfrac\rho R)
\inf_\omega|T_\omega'(x_\omega)|^{k}\big)$, proving 
\ref{it:fslow-}.

Finally, to show \ref{it:fslow}, let $f\in H^2(A_R)^\pm$, and write 
$f=f^++f^-$. Since $Q^+$ and $Q^-$ are orthogonal projections,
both $\|f^+\|$ and $\|f^-\|$ are bounded above by $\|f\|$.
The above shows $\|(I-S_k)\LL_\omega^{(n)}f^-\|\le
c(\lambda_\omega^{(n)})^{k}\|f\|$. Also $(I-S_k)\LL_\omega^{(n)}f^+=
\LL_I(I-S_k)\LL_\omega^{(n)}\LL_If^+$. Since $\LL_I f^+\in H^2(A_R)^-$,
we see 
$$\|\LL_I(I-S_k)\LL_\omega^{(n)}\LL_If^+\|\le\|\LL_I\|^2
c(\lambda_\omega^{(n)})^{k}
\|f\|,
$$
so that the desired conclusion follows by summing the two estimates.

Now let $f\in \mathcal C_{\omega,\eta}$ and let $f=u+v$ where $u\in E_k(\omega)$ and
$v\in F_k(\omega)$. Applying the above inequalities, we have
$$
\|\LL_\omega^{(n)}u\|\ge c_5(\lambda_\omega^{(n)})^k\|u\|\text;
$$
and since $F_k(\omega)\subset H^2(A_R)^\pm$, for $n\ge n_0$, 
$$
\|(1-S_k)\LL_\omega^{(n)}v\|\le c_6(\lambda_\omega^{(n)})^{k}\|v\|.
$$

Notice that $\Pi_{E_k(\sigma^n\omega)\parallel F_k(\sigma^n\omega)}
(1-S_k)\LL_\omega^{(n)}v=
-S_k\LL_\omega^{(n)}v$, so that by Lemma \ref{lem:UH}, 
$\|S_k\LL_\omega^{(n)}v\|\le M\|(1-S_k)\LL_\omega^{(n)}v\|$.
Hence we see
\begin{align*}
\|\LL_\omega^{(n)}v\|&\le (M+1)\|(1-S_k)\LL_\omega^{(n)}v\|\\
&\le c_6(M+1)(\lambda_\omega^{(n)})^{k}\|v\|.
\end{align*}
Now there exists an $n_1$ such that for all $n\ge n_1$ and a.e.\ $\omega$, $c_6(M+1)
\lambda_\omega^{(n)}\le \frac 12c_5$.
Hence if $n\ge \max(n_0,n_1)$, we see that $\LL_\omega^{(n)}f\in
\mathcal C_{\sigma^n\omega,\frac \eta2}$. 
\end{proof}

\subsection{Stability of exponents}

Our theorem in this section is similar to a theorem of Bogensch\"utz
\cite{Bogenschutz}. In his setting, there were two key assumptions: 
uniformity of the splitting and uniformity of convergence to the Lyapunov
exponents. The first of these is satisfied in our setting by the above,
while we relax the second condition by not imposing any convergence conditions on the Lyapunov exponents. 

\begin{lem}\label{lem:Lyapstabzero}
Let $\sigma$ be an ergodic invertible measure-preserving transformation
of $(\Omega,\PP)$, let the Blaschke product cocycle satisfy 
conditions \ref{it:zerofp}, \ref{it:noncrit} and
\ref{it:exp}, and let $\LL_\omega$ be the 
corresponding family of Perron-Frobenius operators.  
For each $\epsilon>0$, let $\LL^\epsilon_\omega$ be a family of
operators such that $\esssup_{\omega\in\Omega}\|\LL^\epsilon_\omega
-\LL_\omega\|\to 0$ as $\epsilon\to 0$. If $(\mu_n)$ are the Lyapunov 
exponents of the unperturbed cocycle listed with multiplicity,
then for each $n$, $\mu_n^\epsilon\to \mu_n$
as $\epsilon\to 0$ where $(\mu_n^\epsilon)$ are the exponents of the perturbed
cocycle.
\end{lem}

\begin{proof}
Let $\Lambda=\int\log|T_\omega'(x_\omega)|\,d\PP(\omega)<0$, so that
by Theorem \ref{thm:LyapSpect}, the 
Lyapunov exponents are $\lambda_j=(j-1)\Lambda$ for $j=1,2,3,\ldots$
where $\lambda_1=0$ has multiplicity 1 and the remaining exponents have
multiplicity 2. List the exponents with multiplicity as $\mu_1=0$, 
and $\mu_{2k}=\mu_{2k+1}=k\Lambda$ for each $k\in\N$.

Let $N$ be as in the statement of Lemma \ref{lem:cone} and $M$ be as in the
statement of Lemma \ref{lem:UH} and $c_5$ be as in the proof of Lemma \ref{lem:cone}.
Let $\eta>0$. Pick $n>N$ so that 
$(\frac 34c_5)^{1/n}>e^{-\eta}$. 
Now  by Corollary \ref{cor:compact}, $\esssup_{\omega\in\Omega}
\|\LL_\omega\|$ is finite. By an application of the triangle inequality, 
$\esssup_{\omega\in\Omega}\|{\LL^\epsilon_\omega}^{(n)}-\LL_\omega^{(n)}\|\to 0$ as
$\epsilon\to 0$. Pick $\epsilon_0>0$ so that $\epsilon<\epsilon_0$ implies
\begin{equation}
\esssup_\omega\|{\LL_\omega^\epsilon}^{(n)}
-\LL_\omega^{(n)}\|<
\frac {c_5\essinf_\omega|T_\omega'(x_\omega)|^{n(k-1)}}{8(M+1)}.
\label{eq:closeness}
\end{equation}

Suppose that $\epsilon<\epsilon_0$ and let $u+v\in \mathcal C_{\omega,1}$,
where $u\in E_k(\omega)$ and $v\in F_k(\omega)$. We claim that for a.e.\ $\omega$,
\begin{align}
{\LL^\epsilon}^{(n)}(u+v)&\in \mathcal C_{\sigma^n\omega,1}\text{; and}
\label{eq:coneinc}\\
\left
\|\Pi_{E_k(\sigma^n\omega)\parallel F_k(\sigma^n\omega)}
{\LL^\epsilon}^{(n)}(u+v)\right\|
&\ge \frac{3c_5}4(\lambda_\omega^{(n)})^{k-1}\|u\|.
\label{eq:lowerbd}
\end{align}

Writing $\Pi_{E\parallel F}^{(n)}$ 
for $\Pi_{E_k(\sigma^n\omega)\parallel F_k(\sigma^n\omega)}$
and $\Pi_{F\parallel E}^{(n)}$ for $I-\Pi_{E\parallel F}^{(n)}$,
we have 
$$
\Pi_{E\parallel F}^{(n)}{\LL^\epsilon_\omega}^{(n)}(u+v)
=\LL_\omega^{(n)}u+
\Pi_{E\parallel F}^{(n)}
({\LL_\omega^\epsilon}^{(n)}-\LL_\omega^{(n)})(u+v),
$$
so that
\begin{align*}
\|\Pi_{E\parallel F}^{(n)}\LL_\omega^{(n)}(u+v)\|
&\ge c_5(\lambda_\omega^{(n)})^{k-1}\|u\|-M\frac{c_5\essinf_\omega
|T_\omega'(x_\omega)|^{n(k-1)}}{8(M+1)}\,2\|u\|\\
&\ge \tfrac 34c_5(\lambda^{(n)}_\omega)^{k-1}\|u\|,
\end{align*}
establishing \eqref{eq:lowerbd}. On the other hand,
$$
\Pi_{F\parallel E}^{(n)}{\LL^\epsilon_\omega}^{(n)}(u+v)
=\LL_\omega^{(n)}v+\Pi_{F\parallel E}^{(n)}({\LL^\epsilon_\omega}^{(n)}
-\LL_\omega^{(n)})(u+v),
$$
so that
\begin{align*}
&\left\|\Pi_{F\parallel E}^{(n)}{\LL^\epsilon_\omega}^{(n)}(u+v)\right\|\\
&\le
\tfrac12c_5(\lambda^{(n)})^{k-1}\|v\|+(M+1)\frac{c_5\essinf_\omega
|T_\omega'(x_\omega)|^{n(k-1)}}{8(M+1)}2\|u\|\\
&\le \tfrac34c_5(\lambda^{(n)})^{k-1}\|u\|.
\end{align*}
This implies ${\LL^\epsilon_\omega}^{(n)}(u+v)\in \mathcal C_{\sigma^n\omega,1}$.

Using \eqref{eq:coneinc} inductively, we see that ${\LL^\epsilon}_\omega^{(mn)}(u+v)
\in \mathcal C_{\sigma^{mn}\omega,1}$ for all $m\in\N$. 
Then an inductive application of \eqref{eq:lowerbd} shows that 
for any $m\in\N$,
\begin{align*}
\|\Pi_{E_k(\sigma^{mn}\omega)\parallel F_k(\sigma^{mn}\omega)}
{\LL_\omega^\epsilon}^{(mn)}(u+v)\|&\ge \left(\frac {3c_5}4\right)^m
(\lambda^{(mn)}_\omega)^{k-1}\|u\|\\
&\ge e^{-\eta mn}(\lambda_\omega^{(mn)})^{k-1}\|u\|.
\end{align*}

Since $(1/mn)\log \big(e^{-\eta mn}(\lambda_\omega^{(mn)})^{k-1}\big)$
converges to $\lambda_k-\eta$ for $\PP$-a.e.\ $\omega$,  
we have identified a $(2k-1)$-dimensional subspace, 
namely $E_k(\omega)$,
on which every vector has Lyapunov exponent at least $\lambda_k-\eta$, so that 
the Lyapunov exponents of the perturbed cocycle, $\mu^\epsilon_j$ (again listed
with multiplicity) satisfy $\mu_{2k-1}^\epsilon>\mu_{2k-1}-\eta$ and
$\mu_{2k-2}^\epsilon\ge \mu_{2k-1}^\epsilon>\mu_{2k-1}-\eta=\mu_{2k-2}-\eta$.
Since $\eta$ and $k$ are arbitrary, this establishes lower semi-continuity of each 
Lyapunov exponent for arbitrary small perturbations of the original cocycle.

To prove the continuity of the exponents for perturbations of the cocycle, it suffices
to show the standard property of upper semi-continuity of the partial sums of the
Lyapunov exponents: that for each $l$ and each $\eta>0$, for all sufficiently small
$\epsilon>0$, one has
$$
\mu_1^\epsilon+\ldots \mu_l^\epsilon<\mu_1+\ldots+\mu_l+\eta.
$$
To show this, define 
$$
\mathcal E_l (\LL)=
\sup_{f_1,\ldots,f_l;\ \phi_1,\ldots,\phi_l
}\det(\phi_i(\LL f_j))_{1\le i,j\le l},
$$
where $f_1,\ldots, f_l$ and $\phi_1,\ldots,\phi_l$ run over the unit sphere
of $H^2(A_R)$ and the unit sphere of the dual space respectively.
The quantity $\mathcal E_l$ is sub-multiplicative (this is standard for Hilbert
spaces, and was demonstrated for arbitrary Banach spaces in \cite{QTZ}).
Results of \cite{GTQ-JMD} combined with the Kingman sub-additive ergodic theorem
show that $\inf_n(1/n)\int \log 
\mathcal E_l(\LL_\omega^{(n)})\,d\PP(\omega)
=\mu_1+\ldots+\mu_l$. In particular, for any $\eta$ and any 
$l$ there exists an $n>0$ such that 
$\frac 1n\int \log \mathcal E_l(\LL_\omega^{(n)})\,d\PP(\omega)
<\mu_1+\ldots+\mu_l+\eta/2$.

For any collections $F=(f_1,\ldots,f_l)$ of functions in the unit sphere of $H^2(A_R)$
and $\Phi=(\phi_1,\ldots,\phi_l)$ of elements of the unit sphere of $H^2(A_R)^*$,
let $E_{F,\Phi}(\LL)=\det(\phi_i(\LL f_j))$. 
These maps are equicontinuous, and indeed uniformly equicontinuous when restricted
to $\{\LL\colon H^2(A_R)\to H^2(A_R)\colon \|\LL\|\le K\}$ for any $K$,
so that $\LL\mapsto\mathcal E_l(\LL)$ is continuous when restricted to 
operators of norm at most $K$. It follows, using Corollary \ref{cor:compact},
that there exists $\epsilon>0$ such that for a.e.\ $\omega\in\Omega$,
$|\mathcal E_l({\LL_\omega^\epsilon}^{(n)})-\mathcal E_l(\LL_\omega^{(n)})|
<\frac\eta 2$.
Hence for sufficiently small $\epsilon>0$,
$$
\int \log \mathcal E_l({\LL_\omega^\epsilon}^{(n)}|_{H^2(A_R)^-})\,d\PP(\omega)
<\mu_1+\ldots+\mu_l+\eta,
$$
so that the sum of the first $l$ Lyapunov exponents of the $\LL_\omega^\epsilon$ cocycle
restricted to $H^2(A_R)^-$ is at most $\mu_1+\ldots+\mu_l+\eta$.
This establishes, for each $l$,
the upper semi-continuity of the sum of the first $l$ Lyapunov exponents
under perturbations of the original cocycle as required.
\end{proof}

We now show that we can deduce Theorem \ref{thm:Lyapcont} as a corollary of the above. 

\begin{lem}\label{lem:powererg}
Let $\sigma$ be an ergodic transformation of $(\Omega,\PP)$ and let $n\in\N$. 
Then there exists $k$, a factor of $n$, and a $\sigma^n$-invariant 
subset $B$ of $\Omega$ of measure $1/k$ such that 
$\Omega=\bigcup_{i=0}^{k-1} \sigma^{-i}B$ and $\sigma^n|_B$ is ergodic.
The ergodic components of $\PP$ under $\sigma^n$ are the restrictions of
$\PP$ to the sets $\sigma^{-i}B$. If $(A_\omega)$ is a matrix or operator
cocycle over $\sigma$, then the cocycle $(A^{(n)}_\omega)$ over $\sigma^n$
restricted to $\sigma^{-i}B$
has Lyapunov exponents $(n\lambda_j)$ where $(\lambda_j)$ are the exponents
of the original cocycle.
\end{lem}

For a proof, one finds the largest factor $k$ of $n$ such that $e^{2\pi i/k}$ is an eigenvalue of the
operator $f\mapsto f\circ\sigma^n$ on $L^2(\Omega)$. The set $B$ is a 
level set of the eigenvector.

For $a\in D_1$, let $M_a$ be the M\"obius transformation $M_a(z)=(z+a)/(1+\bar az)$,
sending $0$ to $a$ and preserving the unit circle, so that in particular these
transformations are Blaschke products. We record without proof the following
straightforward facts about M\"obius transformations.

\begin{lem}\label{lem:Blaschkbounds}
Let $|a|<1$ and let the $M_a$ be as above. Then
\begin{enumerate}[label=(\alph*)]
\item $M_a^{-1}=M_{-a}$\label{it:Minverse}
\item For $z$ in the closed unit disc, $|M_a(z)-z|\le 2|a|/(1-|a|)$. 
In particular if $|a|<\frac 13$, then $|M_a(z)-z|< 3|a|$ 
whenever $|z|\le 1$;
\label{it:Mapproxid}
\item $|M_a'(z)|\le \frac{1+|a|}{1-|a|}$ for all $z$ in the closed unit disc.
\label{it:MLip}
\end{enumerate}
\end{lem}

\begin{proof}[Proof of Theorem \ref{thm:Lyapcont}]
We first prove part \ref{it:stabpert}.
Let $\sigma$, $(\Omega,\PP)$ and $(T_\omega)$ be as in the statement of the theorem
and let $R<1$ satisfy $r:=r_T(R)<R$. Let $x_\omega$ be the random fixed point of 
$(T_\omega)$, as guaranteed by Theorem \ref{thm:LyapSpect}. 
We now define a new conjugate family of cocycles:
$$
\tilde T_\omega=M_{x_{\sigma\omega}}^{-1}\circ T_\omega\circ M_{x_\omega},
$$
so that $\tilde T_\omega(0)=0$ for $\PP$-a.e.\ $\omega$ and $\tilde T_\omega^{(n)}
=M^{-1}_{x_{\sigma^n\omega}}\circ T_\omega^{(n)}\circ M_{x_\omega}$.
If $|z|\le A:=(R-r)/(1-rR)$ then for any $x\in\bar D_r$, $|M_x(z)|\le R$,
so that $T^{(n)}_\omega\circ M_{x_\omega}(\bar D_A)$ is contained in
the intersection of $\bar D_r$ with a disk of radius $c(\frac rR)^n$ 
about $x_{\sigma^n\omega}$. Since the Lipschitz constant
of $M_x$ is $\frac{1+|x|}{1-|x|}$ and 
$M^{-1}_{x_{\sigma^n\omega}}(x_{\sigma^n\omega})=0$,
we see that $\tilde T^{(n)}_\omega(\bar D_A)\subset \bar D_a$ where 
$a=\frac{1+r}{1-r}c(\frac rR)^n$.
Let $n$ be chosen so that $a<A$. 
Since neither $a$ nor $A$ depend on $\omega$, nor does $n$. Let $B$ be an 
ergodic component of $\sigma^n$ as guaranteed by Lemma \ref{lem:powererg}
and consider the cocycle $\tilde\LL_\omega^{(n)}$ restricted to $B$. 
For $\PP$-a.e.\ $\omega$, condition \ref{it:exp} of Lemma
\ref{lem:Lyapstabzero} is satisfied. Condition \ref{it:zerofp} is clearly satisfied
and the above Lipschitz estimates combined with the assumption that
$\essinf_\omega |T_\omega'(x_\omega)|>0$ 
show that $\essinf_\omega |\tilde T^{(n)}_\omega{}'(0)|>0$.
Hence the Lyapunov spectrum for the cocycle $(\tilde\LL^{(n)}_\omega)_{\omega\in B}$,
with base dynamics $\sigma^n\colon B\to B$,
is stable as shown in Lemma \ref{lem:Lyapstabzero}. Together with the final part of Lemma
\ref{lem:powererg}, this implies stability of the 
Lyapunov spectrum for the cocycle $(\tilde\LL_\omega)$ over $\Omega$. 
Since this cocycle is conjugate to the original cocycle:
$\LL_\omega^{(n)}=\LL_{M_{x_{\sigma^n\omega}}}\circ
\tilde\LL_\omega^{(n)}\circ \LL_{M^{-1}_{x_\omega}}$ and
$\LL_{M_x^{\pm1}}$ is uniformly bounded as $x$ runs over $D_r$, we deduce the 
stability of the Lyapunov spectrum of the original cocycle.

For part \ref{it:unstpert}, we first conjugate the Blaschke product cocycle to a 
new cocycle with 0 as the random fixed point.
Write $(\tilde T_\omega)$ for the conjugate Blaschke product cocycle
$\tilde T_\omega(z)=M^{-1}_{x_{\sigma\omega}}\circ T_\omega\circ M_{x_\omega}(z)$,
that is the cocycle where the random fixed point is conjugated to 0, so that $\tilde T_\omega(0)=0$
and $\tilde T_\omega^{(n)}(z)=M^{-1}_{x_{\sigma^n\omega}}\circ T_\omega^{(n)}
\circ M_{x_\omega}(z)$.

The proof will work by modifying the Blaschke product cocycle  $(\tilde T_\omega)$ to give a new
nearby Blaschke product cocycle $(\tilde S_\omega)$ where the random fixed point is still 0.
In the last step, we invert the conjugacy operation to give a new Blaschke product cocycle
$(S_\omega)=(M_{x_{\sigma\omega}}\circ \tilde S_\omega\circ M_{x_\omega}^{-1})$
that has the same random fixed point, $(x_\omega)$, as $(T_\omega)$. 

Notice that since $\tilde T_\omega(0)=0$, $\tilde T_\omega(z)$ may be written as 
$\tilde T_\omega(z)=zP_\omega(z)$ for a rational function $P_\omega$
that is analytic on the unit disc. We see that $P_\omega$ maps the unit circle
to itself so that by Lemma \ref{lem:Blaschkprop}\ref{it:Blaschkchar},
 $P_\omega(z)$ is another Blaschke product. 

We now let $0<\epsilon<1$ and set $\delta=\epsilon (1-r)/(3(1+r))$. 
Define a new family of Blaschke products:
$$
Q_\omega=\begin{cases}
M_{P_\omega(0)}^{-1}\circ P_\omega&\text{if $|P_\omega(0)|<\delta$};\\
P_\omega&\text{otherwise,}
\end{cases}
$$
where the fact that $Q_\omega$ is a Blaschke product follows from 
Lemma \ref{lem:Blaschkprop}\ref{it:Blaschkcomp}.
We can check that $Q_\omega(0)=0$ whenever $|P_\omega(0)|<\delta$. 
By Lemma \ref{lem:Blaschkbounds}, we see that $|Q_\omega(z)-P_\omega(z)|\le 3\delta$ for each 
$z$ in $C_1$. Now set $\tilde S_\omega(z)=zQ_\omega(z)$ so $|\tilde S_\omega(z)-\tilde T_\omega(z)|
\le 3\delta$ for each $z\in C_1$. Next, observe (from the product rule) that $\tilde T_\omega'(0)
=P_\omega(0)$ and $\tilde S_\omega'(0)=Q_\omega(0)$ so that
$\tilde S_\omega'(0)=0$ whenever $|\tilde T_\omega'(0)|< \delta$.
Now
\begin{align*}
\max_{z\in C_1}|S_\omega(z)-T_\omega(z)|
&=\max_{z\in C_1}|M_{x_{\sigma\omega}}\circ\tilde S_\omega
\circ M_{x_\omega}^{-1}(z)-M_{x_{\sigma\omega}}\circ\tilde T_\omega
\circ M_{x_\omega}^{-1}(z)|\\
&=\max_{z\in C_1}|M_{x_{\sigma\omega}}\circ\tilde S_\omega(z)
-M_{x_{\sigma\omega}}\circ\tilde T_\omega(z)|\\
&\le \text{Lip}(M_{x_{\sigma\omega}})\max_{z\in C_1}|\tilde S_\omega(z)-\tilde T_\omega(z)|
\end{align*}

By Lemma \ref{lem:Blaschkbounds} parts 
\ref{it:Minverse} and \ref{it:MLip} and the fact that 
$|x_{\sigma\omega}|\le r$, $\text{Lip}(M^{-1}_{x_{\sigma\omega}})
\le \frac{1+r}{1-r}$ so that $\max_{z\in C_1}|S_\omega(z)-T_\omega(z)|\le\epsilon$.
Hence the new cocycle has the same random fixed point as the old one, 
but for a subset of $\Omega$ of positive measure, we have $S_\omega'(x_\omega)=0$
so that the perturbed cocycle is in case \eqref{it:criticalfp} of Theorem \ref{thm:LyapSpect}
as required.
\end{proof}


\begin{proof}[Proof of Corollary \ref{cor:stabdense}]
We first show that the stable elements of $\Blaschke$ form an open subset. 
Let $\mathcal T\in\Blaschke$ be stable. Let $r=r_{\mathcal T}(R)<R$, let 
$r<\rho<R$, and let $C$ be such that $|w-z|\le d_R(w,z)\le C|w-z|$ for all $w,z$
lying in $\bar{D_\rho}$. 

Now let $\epsilon=\essinf_\omega |T_\omega'(0)|/3$. 
Pick $\delta<\min(\rho-r,(1-r)^2\epsilon,\epsilon (1-\rho)^3(1-\frac r R)/(2C))$
and let $d(\mathcal S,\mathcal T)<\delta$. We first show that the random fixed point, $y_\omega$,
for the $\mathcal S$ cocycle satisfies $d_R(y_\omega,x_\omega)\le C\delta/(1-\frac rR)$. To see this, 
suppose $y$ and $x$ satisfy $d_R(y,x)\le C\delta/(1-\frac rR)$, $T_\omega(\bar D_R)\subset \bar D_r$
and $\max_{z\in C_1}|S_\omega(z)-T_\omega(z)|<\delta$.
Then $d_R(S_\omega(y),T_\omega(x))
\le d_R(S_\omega(y),T_\omega(y))+d_R(T_\omega(y),T_\omega(x))\le
C|S_\omega(y)-T_\omega(y)|+\frac rRd_R(y,x)\le C\delta+\frac rRC\delta/(1-\frac rR)=C\delta/(1-\frac rR)$.
Applying this inductively, we see that for a.e.\ $\omega$, 
$d_R(S_{\sigma^{-n}\omega}^{(n)}(0),T_{\sigma^{-n}\omega}^{(n)}(0))
\le C\delta/(1-\frac rR)$ for all $n$. 
Since for a.e.\ $\omega$, $S_{\sigma^{-n}\omega}^{(n)}(0)\to y_\omega$ and 
$T_{\sigma^{-n}\omega}^{(n)}(0)\to x_\omega$, we see that $d_R(y_\omega,x_\omega)\le C\delta
/(1-\frac rR)$
a.s. Hence $|y_\omega-x_\omega|\le C\delta/(1-\frac rR)$.

Now we have for a.e.\ $\omega$,
\begin{align*}
|S'(y_\omega)|&\ge |T'(x_\omega)|-|S'(x_\omega)-T'(x_\omega)|-|S'(y_\omega)-S'(x_\omega)|\\
&\ge 3\epsilon - \frac \delta{(1-r)^2}-\frac 2{(1-\rho)^3}\, \frac{C\delta}{1-\frac rR}\ge\epsilon,
\end{align*}
where we used Cauchy's formula and the fact that $|x_\omega|\le r$ to estimate
$|(S'-T')(x_\omega)|$ and the fact that $|S''(z)|\le 2/(1-\rho)^3$ on $\bar D_\rho$
for the last term. Hence $\mathcal S$ is stable as required.

Now for the density, let $\mathcal T$ be a Blaschke product cocycle and suppose $r:=r_R(\mathcal T)<R$. 
Write $x_\omega$ for the random fixed point.
We shall obtain a new cocycle $\mathcal S$ with the same random fixed point as $\mathcal T$, in a
similar way to the proof of Theorem \ref{thm:Lyapcont}\ref{it:unstpert}, 
but where we ensure that the derivative
of $S_\omega$ at $x_\omega$ is bounded away from 0. Let $\epsilon<R-r$ and set 
$\delta=\frac{1-r}{6(1+r)}\epsilon$. 

As before, define $\tilde T_\omega=M_{-x_{\sigma\omega}}\circ T_\omega\circ M_{x_\omega}$, so that 
$\tilde {\mathcal T}=(\tilde T_\omega)$ is a conjugate cocycle fixing 0. Since $\tilde T_\omega(0)=0$, 
we have $\tilde T_\omega(z)=zP_\omega(z)$, where $P_\omega$ is another
Blaschke product and $\tilde T_\omega'(0)=P_\omega(0)$. To form the perturbation,
let
$$
Q_\omega=
\begin{cases}
M_{2\delta}\circ P_\omega&\text{if $|P_\omega(0)|\le \delta$;}\\
P_\omega&\text{otherwise.}
\end{cases}
$$
We can check that $|Q_\omega(0)|>\delta$ for each $\omega$. We then let 
$\tilde S_\omega(z)=zQ_\omega(z)$ and $S_\omega(z)=M_{x_{\sigma_\omega}}\circ
\tilde S_\omega\circ M_{-x_\omega}$, so that $S_\omega(x_\omega)=x_{\sigma\omega}$.
As in the proof of Theorem \ref{thm:Lyapcont}\ref{it:unstpert}, we see 
$|S_\omega(z)-T_\omega(z)|\le \epsilon$ for all $z\in C_1$. From the definition of 
$\tilde S_\omega$, 
$|\tilde S_\omega'(0)|>\delta$ and so, using Lemma \ref{lem:Blaschkbounds}\ref{it:MLip}, we see
$|S_\omega'(x_\omega)|>(\frac{1-r}{1+r})^2\delta$ for a.e.\ $\omega$. Hence 
$\mathcal S=(S_\omega)_{\omega\in\Omega}$ is an $\epsilon$-perturbation of $\mathcal T$,
which is stable.
\end{proof}

\subsection{Examples}
For a simple example of a cocycle satisfying the conditions
of Theorem \ref{thm:Lyapcont}, let $a$ and $b$ be any 
numbers in $(0,\frac 13)$
and set
$T_1(z)=[(a+z)/(1+az)]^2$,
$T_2(z)=[(b+z)/(1+bz)]^2$ and $\Omega$ to be the full shift on
$\{1,2\}^\Z$ and $\PP$ any shift-invariant ergodic probability measure,
then one can check that $x_\omega\ge 0$ for all $\omega$
while the critical points are at $-a$ and $-b$. 

Now given any cocycle satisfying the above conditions, one
may apply Theorem \ref{thm:Lyapcont} in each of the following
situations:
\begin{enumerate}
\item(static perturbation)
Suppose for each $\epsilon>0$, 
$(T^\epsilon_\omega)$ is a cocycle of Blaschke
products such that $|T_\omega(z)-
T^\epsilon_{\omega}(z)|\le \epsilon$ 
for each $z\in C_1$. Then by Lemma \ref{lem:PFcont},
$\sup_\omega\|\LL_\om^\epsilon-\LL_\omega\|\to 0$, so that by Theorem
\ref{thm:Lyapcont}, $\mu_n^\epsilon\to \mu_n$ as $\epsilon\to 0$.
\item(quenched random perturbation)
Let $\tau\colon \Xi\to \Xi$ be an invertible measure-preserving transformation such that 
$\sigma\times\tau$ is ergodic. Then consider a family of cocycles of Blaschke products
$(T^\epsilon_{\omega,\xi})$ over the transformation $\sigma\times\tau$ such that
$|T^\epsilon_{\omega,\xi}(z)-T_\omega(z)|\le\epsilon$ for each $z\in C_1$. 
Then the Lyapunov exponents of the perturbed cocycle converge to those 
of the unperturbed cocycle
as $\epsilon$ is shrunk to 0. This follows by considering both 
cocycles as cocycles over $\sigma\times\tau$ (even though the initial cocycle
depends only on the first component). The result then follows by the previous
example.
\item(annealed random perturbation)
Let $\LL^\epsilon_\omega=\mathcal N^\epsilon\circ\LL_\omega$ define a cocycle over 
$\sigma$. We have $\LL_\omega^\epsilon f-\LL_\omega f$ =
$(\mathcal N^\epsilon - I)\LL_\omega f$. By
 the calculations in Lemma  \ref{lem:PFcont}, for any $r<\rho<R$, 
 there exists a $c>0$ such that
 the $e_n$ coefficient of $\LL_\omega f$ is of absolute value at most
$c(\rho/R)^{|n|}\|f\|$. By the calculations in Section \ref{subsec:normalpert},
$\mathcal N^\epsilon-I$ is a diagonal operator with respect to the $e_n$'s, 
scaling $e_n$ by 
$e^{-2\pi^2n^2\epsilon^2}-1\le 2\pi^2n^2\epsilon^2$. 
Hence for any $\omega\in\Omega$ and any $f\in H^2(A_R)$,
the coefficients in the Laurent
expansion of $(\LL^\epsilon_\omega-\LL_\omega)f$ are of absolute value at most
$2\pi^2\epsilon^2c n^2(\rho/R)^{|n|}$. 
In particular, we see $\sup_{\omega\in\Omega}\|\LL^\epsilon_\omega
-\LL_\omega\|\to 0$ as $\epsilon\to 0$, so that we can apply 
Theorem \ref{thm:Lyapcont}. Hence the Lyapunov exponents
vary continuously, in strong contrast to the situation 
in Section \ref{subsec:normalpert} involving applying the same perturbations to 
another cocycle.
\end{enumerate}

\appendix
\section{Comparison of Lyapunov spectrum and Oseledets splittings on 
different function spaces}\label{S:comparison}
Let $\mathcal R=(\Omega,\PP,\sigma,X,\LL)$ be a random linear dynamical system
and suppose that $X'$ is a dense subspace of $X$, equipped with a norm $\|\cdot\|_{X'}$
such that $\|x\|_{X'}\ge \|x\|_X$ for all $x\in X'$ and
$\LL_\omega(X')\subset X'$.
Then we say $\mathcal R'=(\Omega,\PP,\sigma,X',\LL|_{X'})$ is a 
\emph{dense restriction} of $\mathcal R$. We restate Theorem \ref{thm:METcomp}
more precisely in this language.

\begin{thm}[Comparison of Lyapunov exponents Oseledets splittings]\label{thm:SameSp}
Let $\mc{R}=(\Om, \bbp, \sig, X, \mcl)$ be a random linear dynamical
system with ergodic invertible base and let $\mc{R'}$ 
be its dense restriction to the Banach space $X'$.
Suppose the two systems satisfy the assumptions of Theorem \ref{thm:MET}.


Let $X=W(\om)\oplus \bigoplus_{j=1}^l V_j(\om)$ and $X'=W'(\om)\oplus \bigoplus_{j=1}^{l'} V'_j(\om)$
be the splittings associated to $\mc{R}$ and  $\mc{R}'$, respectively, and let 
$\{\lam_j\}_{1\leq j \leq l}$ and $\{\lam'_j\}_{1\leq j \leq l'}$ be the corresponding 
exceptional Lyapunov exponents.
Then, whenever $\max(\lam_j,\lam'_j)>\alpha:=\max(\ka(\mc R),\kappa(\mathcal R'))$,
\begin{enumerate}
  \item \label{it:exponents} $\lam_j=\lam'_j$; and
\item \label{it:spaces} For $\bbp$-a.e. $\om$, $V_j(\om)= V'_j(\om)$.
\end{enumerate}
\end{thm}

For each $\om\in \Om, f\in X$, we let $\lex{\om, f}{X}=\limsup_{n\to \infty}
\frac 1n \log \|\mcl_\om^{(n)}f\|_X$. If $f\in X'$, we define $\lex{\om, f}{X'}
=\limsup_{n\to \infty}\frac 1n \log \|\mcl_\om^{(n)}f\|_{X'}$.
The following result will be needed in the proof. 

\begin{lem}[Coincidence of external exponents] \label{lem:AgainSameExp}
Let $\mc{R}$, $\mc{R'}$ and $\al$ be as in Theorem~\ref{thm:SameSp}. 
Then, for every $f\in X'$ for which $\lex{\om, f}{X'}> \al$
 one has that $\lex{\om, f}{X}=\lex{\om, f}{X'}$. 
\end{lem}

\begin{rem}
A slightly different result was established in  \cite[Theorem 3.3]{StancevicFroyland}. 
We include a proof of Lemma~\ref{lem:AgainSameExp} for completeness.
\end{rem}
%

\begin{proof}
Let $f\in X'$ satisfy $\lambda_{X'}(\omega,f)>\alpha$. Clearly 
$\lambda_{X}(\omega,f)\le \lambda_{X'}(\omega,f)$. 
Since $\lambda_{X'}(\omega,f)>\alpha$, there exists a $j$
such that $\lambda_{X'}(\omega,f)=\lambda'_j$. Write 
$V'_j(\omega)$ for the corresponding Oseledets subspace 
of $X'$ and note that $V'_j(\omega)$ is also a subspace of $X$. 
Since $V'_j(\omega)$ is finite-dimensional,
there exists a positive measurable function $c(\omega)$ such that 
$\|g\|_{X}\ge c(\omega)\|g\|_{X'}$ for all $g\in V'_j(\omega)$. 
Now write $f=g+h$ with $g\in V'_j(\omega)$ and $h\in F_j'(\omega)$. 
We have
$$
\|\LL^{(n)}_\omega f\|_X\ge \|\LL^{(n)}_\omega g\|_X-
\|\LL^{(n)}_\omega h\|_X\ge \|\LL^{(n)}_\omega g\|_{X'}\,c(\sigma^n\omega)
-\|\LL^{(n)}_\omega h\|_{X'}
$$
Taking a limit along a positive density sequence of $n$'s where 
$c(\sigma^n\omega)$ is bounded away from 0, we see that $\lambda_X(\omega,f)\geq
\lambda_{X'}(\omega,f)$, so that the two quantities agree.
\end{proof}

\begin{proof}[Proof of Theorem~\ref{thm:SameSp}]
We prove the result by induction. 
Suppose $\lambda_i=\lambda'_i$ and $V_i(\omega)=V_i'(\omega)$ for
$\PP$-a.e.\ $\omega$ for $i=1,\ldots,j-1$, with $j\geq 1$. 
If $\lambda'_j>\alpha$, then letting $f\in V_j'(\omega)$ and applying the
lemma above, we see that $\lambda_j\ge \lambda'_j$. 


Using continuity of $\Pi_{F_{j-1}(\omega)\parallel E_{j-1}(\omega)}$
and the induction hypothesis, $F_{j-1}(\omega)\cap X'=
\Pi_{F_{j-1}(\omega)\parallel E_{j-1}(\omega)}(X')$ is dense in $F_{j-1}(\omega)$
and therefore $\Pi_{E_j(\omega)\parallel F_j(\omega)}(X'\cap F_{j-1}(\omega))=
V_j(\omega)$. Let $U$ be a subspace of $X'\cap F_{j-1}(\omega)$ of dimension $m_j=\dim V_j(\omega)$
such that $\Pi_{E_j(\omega)\parallel F_j(\omega)}(U)=V_j(\omega)$. Now if
$h\in U\setminus\{0\}$, then 
$\lambda_{X'}(\om,h)=\lambda_X(\om,h)=\lambda_j$. 

It follows that $\lambda_j$ is
an exceptional exponent of $\mathcal R'$ and
$\lambda_{j}'\ge \lambda_j$, so that $\lambda_j=\lambda_{j'}$.
Since $U$ is $m_j$-dimensional, we claim that $V'_j(\omega)$ is of dimension
at least $m_j$. To see this, notice that if not, there would be a non-zero element
of $U$ whose projection under $\Pi_{V'_j(\omega)}$ would be trivial, so
that the growth rate of this element would be strictly smaller than
$\lambda_j$, giving a contradiction.
By Lemma~\ref{lem:AgainSameExp}, we see that $V'_j(\omega)\subset V_j(\omega)$
and the above argument shows that $\dim V'_j(\omega)\ge \dim V_j(\omega)$,
so that $V_j'(\omega)=V_j(\omega)$ as required.
\end{proof}

\subsection{Example}
We consider finite Blaschke products of the form $B(z)= z \prod_{j=1}^n\frac{z+\zeta_j}{1+\bar{\zeta}_jz}$.
Note that $B(0)=0$ and $B'(0)=\prod_{j=1}^n \zeta_j$.
Furthermore, \cite[Proposition 1]{Martin} ensures that $\inf_{|z|=1} |B'(z)|
\geq 1+ \sum_{j=1}^n\frac{1-|\zeta_j|}{1+|\zeta_j|}>1$. 

For $0<a<1$, let $B_a(z)= z \big(\frac{z-a}{1-a z} \big)^2$.
Notice that $B_{a}'(0)=a^2$ and $\inf_{|z|=1} |B_a'(z)|\geq 1+\frac{2(1-a)}{1+a}$.
Let $\Omega=\{0,1\}^\Z$, $\sigma$ be the shift map,
and $\PP$ be the Bernoulli measure with $\PP([0])=0.5$.
Let $\LL_0$ be the Perron-Frobenius operator of $B_{0.5}$,  $\LL_1$ be the 
Perron-Frobenius operator of $B_{0.6}$
and consider the operator cocycle 
$\mathcal R=(\Omega,\PP,\sigma,X,\LL)$
generated by $\LL_\omega:=\LL_{\omega_0}$, 
acting on $X=C^3(S^1)$, as well as the dense restriction
$\mathcal R'=(\Omega,\PP,\sigma,X',\LL|_{X'})$, where $X'=H^2(A_R)$
(since both $B_{0.5}$ and $B_{0.6}$ are expanding, $R$ may be chosen 
so that $r_{B_a}(R)<R$ for $a\in \{0.5,0.6\}$).
By Theorem~\ref{thm:LyapSpect}, the Lyapunov spectrum of $\mathcal R'$ is 
$\Sigma(\mathcal{R}')=\{0\} \cup \{n\Lambda :n\in \N\}$, where 
$\Lambda= \log (0.5) + \log(0.6) > -1.204$.

Note that $\inf_{|z|=1, a\in\{0.5,0.6\}} |B_a'(z)|\geq 1+2\frac{0.4}{1.6}=1.5=:\gamma$.
The work of Ruelle \cite{Ruelle} ensures the essential spectral radius of each of 
$\mcl_0, \mcl_1$ acting on the space of $C^r$ functions satisfies 
$\rho_e(\mcl_B)\leq 1/\gamma^r$.
This result may also be established relying on Lasota-Yorke type inequalities, 
for example, following the strategy presented in \cite[Lemma 3.3 \&
Corollary 3.4]{Gouezel-LimThms} for the $C^1$ case. A random version of Ruelle's 
result hence follows from \cite[Lemma C.5]{GTQ-ETDS}, which provides a random 
version of Hennion's theorem. 
Thus, we have that $\kappa(\mathcal R) \leq -r \log \gamma= -3\log 1.5 <-1.21<\Lambda<0$.
Theorem~\ref{thm:SameSp} implies that $\Lambda$ is an exceptional Lyapunov exponent of $\mathcal R$.

We would like to acknowledge support from the Australian Research Council (DE160100147) (CGT)
and NSERC (AQ). We would like to thank Universities of Queensland and Victoria for hospitality. 
Our research benefitted from a productive environment at the 2018 Sydney Dynamics Group Workshop. We
would also like to thank Mariusz Urba\'nski and Matteo Tanzi for useful suggestions. 

\bibliographystyle{abbrv}
\bibliography{PFCollapse}

\end{document}